\documentclass[12pt]{article}

\usepackage{amsfonts,amsmath,amsthm,bm,bbm, thmtools, thm-restate}
\usepackage{latexsym,color,epsfig,mathrsfs,enumerate}
\usepackage[toc,page]{appendix}
\usepackage{fancyhdr}
\usepackage[colorlinks=true, allcolors=red]{hyperref}
\usepackage{indentfirst}
\usepackage[maxbibnames=9]{biblatex}
\usepackage{comment}
\usepackage{xcolor}
\usepackage[shortlabels]{enumitem}

\usepackage{graphicx}
\usepackage{pgfplots}
\pgfplotsset{compat=1.15}
\usetikzlibrary{arrows}

\nocite{*}

\newtheorem{thm}{Theorem}[section]
\newtheorem{dfn}[thm]{Definition}
\newtheorem{lem}[thm]{Lemma}
\newtheorem{cor}[thm]{Corollary}
\newtheorem{prop}[thm]{Proposition}

\newtheorem{conjecture}[thm]{Conjecture}

\newcommand{\laur}[1]{\textbf{\underline{#1}}}
\newcommand{\bad}[1]{\operatorname{bad}_{\mathcal{D}}({#1})}
\newcommand{\baddd}[4]{\operatorname{bad}_{\mathcal{#1}_{#2}}^{#3}({#4})}
\newcommand{\proj}[3]{\operatorname{proj}_{#1}({\bf {#2}}_{#3})}
\newcommand{\diver}[1]{\operatorname{div}(#1)}

\usepackage{combelow}

\usepackage{newunicodechar}
\newunicodechar{ș}{\cb{s}}
\newunicodechar{ț}{\cb{t}}

\newcommand{\ov}{\overline}
\newcommand{\ul}{\underline}

\newcommand{\C}{\mathcal{C}}

\newcommand{\R}{\mathbb R}

\newcommand{\bP}{\mathbb P}
\newcommand{\bE}{\mathbb E}

\title{Distinct degrees and homogeneous sets}
\author{Eoin Long\thanks{School of Mathematics, University of Birmingham, UK. Email: \texttt{e.long@bham.ac.uk}.} \and Lauren\cb{t}iu Ploscaru\thanks{School of Mathematics, University of Birmingham, UK. Email: \texttt{ixp090@student.bham.ac.uk}.\newline \hspace*{1.5em} 
The second author is grateful for support through an EPSRC DTP studentship.}}
\date{13 April 2022}
\begin{document}

\maketitle

\begin{abstract}
    In this paper we investigate the extremal relationship between two well-studied graph parameters: the order of the largest 
    homogeneous set in a graph $G$ and the maximal number of distinct degrees appearing in an induced 
    subgraph of $G$, denoted respectively by $\hom (G)$ and $f(G)$.\vspace{1mm}
    
    Our main theorem improves estimates due to several earlier researchers and 
    shows that if $G$ is an $n$-vertex graph with $\hom (G) \geq n^{1/2}$ then 
    $f(G) \geq \big ( {n}/{\hom (G)} \big )^{1 - o(1)}$. The bound here is sharp up to the $o(1)$-term, and 
    asymptotically solves a conjecture of Narayanan and Tomon. In particular, this implies that $\max \{ \hom (G), f(G) \} \geq n^{1/2 -o(1)}$ for any $n$-vertex graph $G$, 
    which is also sharp.\vspace{1mm}
    
	The above relationship between $\hom (G)$ and $f(G)$ breaks down in the regime where $\hom (G) <  n^{1/2}$. Our second result 
	provides a sharp bound for distinct degrees in biased random graphs, i.e. on 
	$f\big (G(n,p) \big )$. 
	We believe that the behaviour here determines the extremal relationship between 
	$\hom (G)$ and $f(G)$ in this second regime.\vspace{1mm}
    
    Our approach to lower bounding $f(G)$ proceeds via a translation into an (almost) 
    equivalent probabilistic problem, and it can be shown to be effective for arbitrary 
    graphs. It may be of independent interest.
    \end{abstract}

\section{Introduction }
\label{section: introduction}

The focus of this paper is on the extremal relationship between the order of the largest homogeneous set in a graph $G$ and the maximal number of distinct degrees which appear in some induced subgraph of $G$. More precisely, let $\hom (G)$ denote the \emph{homogeneous number} of a graph $G$, given by: 
	\begin{align*}
		\hom (G) 
			&:= 
		\max \big \{ |U| : U \subset V(G) \mbox { with } G[U] \mbox { a complete or empty graph} 
		\big \}. 
	\end{align*}
\indent We also let $f(G)$ denote the \emph{distinct degree number of} $G$, given by: 
	\begin{align*}
		f(G) 
			&:=
		\max \big  \{ k \in {\mathbb N}: G[S] \mbox{ has } k \mbox{ distinct degrees for some } 
		S \subset V(G) \big  \}.
	\end{align*}
\indent These quantities have been well-studied in the literature. Indeed, $\hom (G)$ arises as a key parameter in a variety of settings, including extremal graph theory, graph Ramsey theory and perfect graph theory (see for example \cite{bollo}, \cite{spencert}, \cite{ramsey-survey}, \cite{ramírez2001perfect}). On the other hand, a wide range of results aim to study the possible degree distributions of (induced) subgraphs of a graph, for example \cite{lovasz1993combinatorial}, \cite{Alon198479}, \cite{Pyber199541}, \cite{scott_1992}, \cite{ferber2021graph}, \cite{longetcomp}, and $f(G)$ arises very naturally in this context. 

Erd\H{o}s, Faudree and S\'os  were the first to investigate the relationship between $\hom (G)$ and $f(G)$, focusing in particular on the Ramsey setting, where $\hom (G)$ is (essentially) minimal. Recall that Ramsey's theorem \cite{ramsey1930}, \cite{Erdosold} guarantees that every $n$-vertex graph $G$ satisfies the relation $\hom (G) = \Omega (\log n)$. Erd\H{o}s \cite{Erdosprob} showed, in what is one of the earliest instances of the probabilistic method \cite{alon04}, that there are $n$-vertex graphs $G$ with $\hom (G) = \Theta (\log n)$ and so the logarithmic order is sharp here. However, the existence of all such graphs $G$ has only been demonstrated indirectly via some random process and it is a major open problem to give explicit examples of such graphs (see \cite{barak20122}, \cite{xinli}). Motivated by this, a large body of research has developed concerning the structure of Ramsey graphs \cite{erdosmrd}, \cite{PROMEL}, \cite{Shelah}, \cite{kwan2018ramsey}, \cite{KSP},\cite{narayanan2017ramsey}, aiming to show that they must behave similarly to appropriate random graphs. 

In this context, Erd\H{o}s, Faudree and S\'os \cite{sos} noticed that the random graph $G(n,1/2)$ has $f(G(n,1/2)) = \Omega (n^{1/2})$ with high probability. They conjectured this property must be shared by Ramsey graphs: if $G$ is an $n$-vertex graph with $\hom (G) = O(\log n)$ then $f(G) = \Omega (n^{1/2})$. Bukh and Sudakov confirmed this conjecture in \cite{bukh} with an elegant and influential argument. Furthermore, they noted that there still appeared to be some flexibility here: 
\begin{itemize}
	\item [(a)] Although $f(G(n,1/2)) = \Omega (n^{1/2})$ forms a natural lower bound, 
	they observed that 
	it did not have a matching upper bound, as they proved that $f(G(n,1/2)) = O(n^{2/3})$ whp. 
	\item [(b)] They conjectured that $\hom (G) = n^{o(1)}$ already implies 
	that $f(G) \geq n^{1/2-o(1)}$.
\end{itemize}

 It was later shown by Conlon, Morris, Samotij and Saxton \cite{unpublished}, thus matching the upper bound given in (a), that in fact $f(G(n,1/2)) = \Omega (n^{2/3})$ whp. Recently, Jenssen, Keevash, Long and Yepremyan \cite{JKLY} proved that the same lower bound applies in the Ramsey context, giving a tight bound for the original Ramsey question of Erd\H{o}s, Faudree and S\'os. 

In \cite{narayanan}, Narayanan and Tomon solved the conjecture from (b) above, proving that actually $f(G) = \Omega \big ( (n/\hom (G))^{1/2} \big )$ for all $n$-vertex graphs $G$. They also provided an interesting construction, which suggested a tight bound between the following parameters: if $k \leq n^{1/2}$ then the $n$-vertex $k$-partite Tur\'an graph $T$ (see e.g.\cite{bollo}) satisfies both $\hom (T) = n/k$ and $f(T) = k$. Narayanan and Tomon conjectured that a similar dependence must hold in general: if $G$ is an $n$-vertex graph satisfying $\hom (G) \geq n^{1/2}$ then $f(G) = \Omega (n / \hom (G))$. Supporting their conjecture, the authors proved that indeed $f(G) = \Omega (n / \hom (G))$ when $\hom (G) = \Omega (n/ \log n)$. Jenssen et al. \cite{JKLY} improved this bound to $\hom (G) \geq n^{9/10}$, noting that there were significant obstacles to obtaining $\hom (G) \geq n^{1/2}$.

Our main result here confirms the Narayanan$-$Tomon conjecture up to a logarithmic loss.\vspace{1mm}

\begin{thm}
    \label{thm: N-T-intro-version}
      Let $m\geq n^{1/2}$. Then every $n$-vertex graph $G$ with $\hom (G)\leq m$ satisfies:
      $$f(G) = \Omega \bigg ( \frac {n/m}{\log ^{7/2} (n/m)} \bigg ).$$
\end{thm}

As an immediate corollary of Theorem \ref{thm: N-T-intro-version} we obtain the following result, which 
strengthens the bounds of Bukh and Sudakov \cite{bukh} and of Narayanan and Tomon \cite{narayanan}. 

\begin{cor}
	\label{cor: concise max version}
	Every $n$-vertex graph $G$ satisfies $\max \big \{\hom (G), f(G) \big \} \geq n^{1/2 - o(1)}$.
\end{cor}

 Again, note that the $n^{1/2}$-partite Tur\'an graph on $n$ vertices shows that this bound is essentially sharp. However, as discussed below, there is a large and varied collection of graphs which are close to extremal value here.

Our second result focuses on the regime where $\hom (G) < n^{1/2}$. The Tur\'an construction given above begins to break down here, and in fact the above relationship between the parameters no longer holds; e.g. by our discussion above $\hom (G(n,1/2))\cdot  f(G(n,1/2)) = \Theta (n^{2/3} \log n) \ll n$ whp. Motivated by this, we prove sharp bounds on $f(G(n,p))$ for general values of $p$, extending the results of Bukh and Sudakov \cite{bukh} and of Conlon, Morris, Samotij and Saxton \cite{unpublished}. 

\begin{restatable}{thm}{randomdistinctdegrees}
\label{thm: DD for G(n,p)}
Let $n \in {\mathbb N}$ and let $p:= p(n) \in [0,1/2]$. Then whp the random graph $G(n,p)$ satisfies the following: 
\begin{itemize}
	\item [\emph{(i)}] $f\big(G(n,p)\big)=\Theta\left(\sqrt[3]{pn^2}\right)$ for $p \in [n^{-1/2},1/2]$;
	\item [\emph{(ii)}] $f\big (G(n,p)\big ) = \Theta\big (\Delta ( G(n,p) ) \big )$ for $p \leq n^{-1/2}$.  
\end{itemize}
\end{restatable}

\noindent \emph{Remark.} As $f(G) = f(\overline{G})$ for any graph $G$, we see that $f(G(n,p))$ and $f(G(n,1-p))$ follow identical distributions, so Theorem \ref{thm: DD for G(n,p)} determines the behaviour of $f(G(n,p))$ for all $p \in [0,1]$.

Together with the known estimates on the homogeneous number of sparse random graphs, Theorem \ref{thm: DD for G(n,p)} suggests a natural extremal relationship between $\hom (G)$ and $f(G)$ when the hypothesis of Theorem \ref{thm: N-T-intro-version} fails, i.e. when $\hom (G) < n^{1/2}$. We further examine this relationship in the concluding remarks in Section {\ref{section: concluding-remarks}}.

Our proofs to both Theorem \ref{thm: N-T-intro-version} and Theorem \ref{thm: DD for G(n,p)} build upon earlier approaches from \cite{bukh} and \cite{JKLY}, but there are many extra challenges in this regime, which require several key new ingredients and ideas. For instance, although Tur\'an graphs represent an example of $n$-vertex graphs $G$ with $\hom (G) = n^{1/2}$ and $f(G) = \Theta (n^{1/2})$, there are several very different looking graphs which exhibit (essentially) the same behaviour, including the random graph $G(n,n^{-1/2})$. 

One interesting class of examples was given by Narayanan and Tomon, which we call `iterated Tur\'an graphs': take $b < n^{1/2}$ vertex disjoint sets $V_1,\ldots, V_b$ of size $n/b$, and on each set $V_i$ put a copy of the complement of the $n^{1/2}$-partite $n/b$-vertex Tur\'an graph, and join all pairs lying in distinct $V_i$ and $V_j$ by an edge. It can be checked that such a graph $G$ has $n$-vertices, that $\hom (G) = n^{1/2}$ and that $f(G) = n^{1/2}$ (any set $V_i$ contains at most $n^{1/2}/b$ vertices with distinct degrees). Noting that the degree in each such graph is $n - n/b + n^{1/2}/b$, we see these graphs are non-isomorphic for different values of $b$, and so there are many distinct extremal situations. 

Standing back from this, consider (1) starting with many vertex disjoint copies of the same graph, (2) complementing the edges of each, and (3) joining all vertices between different classes by an edge. One can observe that, starting with the graph on a single point and running these steps we can obtain a Tur\'an graph (applying the process once), and the iterated Tur\'an graph (applying it twice). One could furthermore iterate more times, and this leads to graphs with very limited neighbourhood diversity (see  the definition before Lemma \ref{lem: blended-distribution-control} below), which was a key parameter in many earlier approaches. One of our results below (see Theorem \ref{thm: DExpD}) allows us to prove lower bounds on $f(G)$ by instead lower bounding auxilliary parameters (see Theorem \ref{thm: DExpD}) and this connection crucially works without diversity assumptions, unlike in earlier approaches.

The above process also highlights a more significant challenge, which arises naturally for this problem. To find a large set $U$ of vertices with distinct degrees in general, these iterated graphs show that sometimes we \emph{must} first find sets $U_i$ of distinct degrees locally in smaller graphs and then combine the results into a larger set $U = \cup U_i$. Combining such sets together can work very well for iterated graphs, but even small changes to the structure here can break the condition $-$ at an extreme, it could be that the sets $U_i$ have distinct degrees in $G[V_i]$ for $i =1,2$ with $V_1$ and $V_2$ disjoint, but that all vertices of $U_1 \cup U_2$ have the same degree when combined in $G[V_1 \cup V_2]$. We avoid this kind of difficulty by moving to a more general probabilistic setting, where we instead find probability distributions with certain well-controlled small ball probabilities. 

Lastly, our approach in Sections \ref{sec: degrees and distributions on the cube} and \ref{sec: building distributions} is quite applicable to the general problem of lower bounding $f(G)$ in an arbitrary graph $-$ see Theorem \ref{thm: DExpD} and Lemma \ref{lem: reverse direction} below.

The paper is organised as follows. In the next section we present a number of tools which will be required in our proof. In Section \ref{sec: degrees and distributions on the cube} we present a probabilistic analogue of the problem of finding many distinct degrees in a graph. In Section \ref{sec: building distributions} we extend this approach to a more robust variant and develop a variety of tools and estimates for studying the distinct degree problem. In Section \ref{sec: N-T conjecture} we prove Theorem \ref{thm: N-T-intro-version} as follows: we first deal with a slightly weaker version in Section \ref{subsection: proof in the larger regime}, which applies when $\hom (G) \geq n^{3/5 + o(1)}$, and then build upon this in subsection \ref{subsection: proof in the entire regime} to prove Theorem \ref{thm: N-T-intro-version}. In Section \ref{sec: random graph distinct degrees} we present the proof of Theorem \ref{thm: DD for G(n,p)}. Finally, in Section \ref{section: concluding-remarks} we conclude with a discussion of the case when $f(G) < n^{1/2}$.

\noindent \textbf{Notation.} Given a graph $G$ and $u,v\in V(G)$, we write $u\sim v$ if $u$ and $v$ are adjacent vertices in $G$ and $u\not\sim v$ if they are not. The neighbourhood of $u$ is given by $N_G(u) = \{v\in V(G): u\sim v\}$ and given $S \subset V(G)$ we let $N^S_G(u) := N_G(u) \cap S$; we will omit the subscript $G$ when the graph is clear from the context. We write $d^S_G(u) = |N^S_G(u)|$.

\indent Given a vertex $u$, we will also represent the neighbourhood of $u$ by a vector $\textbf{u} \in \{0,1\}^{V(G)}$ defined such that $u_v=1$ if and only if $u\sim v$. Given a set $U \subset V$ and a vector ${\bf u} \in {\mathbb R}^V$, we will denote the projection of ${\bf u}$ onto the coordinate set $S$ by $\proj {S}{u}{}$, i.e. for any $v \in S$ we have $ \proj {S}{u}{} _v = {\bf u} _v$. Given $u,v\in V(G)$ we write $\text{div}_G(u,v)$ for the symmetric difference $N(u)\triangle N(v)$. Thus $|\text{div}_G(u,v)|$ is simply the Hamming distance between ${\bf u}$ and ${\bf v}$.

\indent We will write $\ov{G}$ for the complement of the graph $G$. It is easy to note that for any graph $G$ we have $\hom (G) = \hom (\ov{G})$ and $f(G)=f(\ov{G})$ since $\text{div}_G(u,v)=\text{div}_{\ov{G}}(u,v)$ for any $u,v\in V(G)$.

\indent Given $n\in \mathbb{N}$ and $p\in (0,1)$, the Erd\H{o}s$-$R\'enyi random graph $G(n,p)$ is the $n$-vertex graph in which each edge is included in the graph with probability $p$ independently of every other edge. We say that an event that depends on $n$ occurs \emph{with high probability} (whp) if its probability tends to $1$ as $n\to \infty$.

\indent Throughout this paper we will omit floor and ceiling signs when they are not crucial, for the sake of clarity of presentation.

\section{Tools}

In this short section we introduce some tools required for the rest of the paper. We will use the following version of Tur\'an's theorem (see for example Chapter 6 in \cite{bollo}).

\begin{thm}\label{turan}
Let $G$ be a $n$-vertex graph with average degree $d$. Then $G$ has an independent set of size at least $n/(d+1)$. 
\end{thm}

Secondly, we require the following `anticoncentration' theorem for the Littlewood$-$Offord problem, which is due to Erd\H{o}s \cite{LO}:

\begin{thm}[\textbf{Erd\H{o}s--Littlewood--Offord}]\label{loff}
Let $S$ be a set of $n$ real numbers of absolute value at least $1$. Then, for each $\alpha\in \mathbb{R}$, there are at most $\binom{n}{\lfloor n/2\rfloor} = \Theta (2^n n^{-1/2})$ subsets of $S$ whose sum of elements lie in the interval $[\alpha,\alpha+1)$.

\end{thm}

We now give a probabilistic interpretation of the previous theorem, which can also be found in \cite{JKLY}. For the sake of completeness, we include a proof of this result.

\begin{thm}\label{problitoff}
Fix non-zero parameters $a_1,a_2,\dots,a_n\in \mathbb{R}$ and $p_1,p_2,\dots,p_n\in [0.1,0.9]$. Suppose that $X_1,X_2,\dots, X_n$ are independent Bernoulli random variables with $X_i\sim Be(p_i)$. Then: $$\displaystyle\max_{x\in \mathbb{R}}\ \bP \bigg ( \displaystyle\sum_{i=1}^n a_iX_i=x \bigg ) =O(n^{-1/2}).$$
\end{thm} 

\begin{proof} For each $i\in [n]$ choose $w_i,z_i\in [0,1]$ such that $p_i=w_i/2+(1-w_i)z_i$. Then write $X_i$ as $X_i=W_iY_i+(1-W_i)Z_i$, where $W_i\sim Be(w_i),\ Z_i\sim Be(z_i)$ and $Y_i\sim Be(0.5)$ are independent random variables. We want to make this choice so that each $w_i\geq 0.2$ and we can do this by letting $z_i=0, w_i=2p_i$ if $p_i\leq 1/2$ and by letting $z_i=1, w_i=2(1-p_i)$ if $p_i>1/2$. 

We now condition on any choice $\C$ of the $W_i$’s and $Z_i$’s. Let $I=\{i\in [n]:W_i=1\}$ and suppose that we have $Z_i = z_i$ after the conditioning. On one hand, if $|I|\geq n/10$ then $\bP \big ( \sum_{i=1}^n a_iX_i=x\ |\ \C \big ) = \bP \big ( \sum_{i\in I} a_iY_i + \sum _{i \notin I} a_i z_i =x\ |\ \C \big )$ becomes $\bP \big ( \sum_{i \in I} a_iY_i =x_{\C} \big )$ where $x_{\C} = x - \sum _{i \notin I} z_i$ is a constant, which by Theorem \ref{loff} (eventually with a scaling argument) is at most $O(n^{-1/2})$. On the other hand, let $\ov{W}=(W_1+W_2+\dots+W_n)/n$ and observe that $|I|=n\ov{W}$. Moreover, we have $\bE[\ov{W}]\geq 0.2$ since $w_i\geq 0.2$ for each $i$. Therefore we deduce $\bP(|I|\leq n/10)=\bP(\ov{W}\leq 0.1)\leq \bP(\ov{W}-\bE[\ov{W}]\leq -0.1)\leq \bP(|\ov{W}-\bE[\ov{W}]|\leq -0.1)$. So we get by Chebyshev's Inequality that $\bP(|I|\leq n/10)\leq O(n^{-1})$. 

The conclusion follows by combining these two results in the Law of Total Probability.
\end{proof}

The following optimization results will be very useful along the way.

\begin{lem}\label{oprimization-map}
Let $b\geq a>0$ and let $0<\alpha<1$. Then the function $f:[0,a)\to\mathbb{R}$ given by $f(x):=(b+x)^\alpha + (a-x)^\alpha$ is strictly decreasing. In particular, for all $t\in (0,a)$ we have:
$$b^\alpha+a^\alpha> (b+a-t)^\alpha + t^\alpha .$$
\end{lem}
\begin{proof}
Note that $f^\prime(x)=\alpha(b+x)^{\alpha-1}-\alpha(a-x)^{\alpha-1}<0$ on the interval $[0,a)$ since $\alpha-1<0$ and $b+x\geq a-x>0$. This gives us the first part, whereas the second one is just $f(0)>f(a-t)$.
\end{proof}

\begin{lem}\label{log2ineq}
Let $a,b>0$. Then $a\log_2a+b\log_2b+2\min\{a,b\}\leq (a+b)\log_2(a+b)$. 
\end{lem}
\begin{proof}
We may assume that $a\leq b$. Let now $x:=a+b$ and $a:=tx$ with $0<t\leq 1/2$. Upon dividing by $x$, the inequality we need to prove rewrites as $2t+t\log_2 (tx)+(1-t)\log_2((1-t)x)\leq \log_2 x$, which is equivalent to $2t+t\log_2 t+(1-t)\log_2(1-t)\leq 0$ for $0<t\leq 1/2$.\\
\indent The map $f:(0,\infty)\to\mathbb{R}$ given by $f(y)=y\log_2 y$ is convex and can be continuously extended to $f(0)=0$. Therefore the $LHS$ in our last inequality above is convex, so we only need to check that the inequality holds for $t=0$ and $t=1/2$, which can be easily seen.
\end{proof}

\vspace{6pt}
Finally, we require some classic concentration inequalities. See e.g. appendix A in \cite{alon04}.

\begin{thm}[\textbf{Chernoff Inequality}]\label{cher}
Let $X$ be a random variable with binomial distribution and let $\mu = \bE[X]$. Then, for $0\leq \delta\leq 1$, the following inequalities hold:
$$ \bP\big(X\leq (1-\delta )\mu \big)\leq \exp\left(-{\dfrac {\delta ^{2}\mu }{2}}\right).$$
$$ \bP\big(X\geq (1+\delta )\mu \big)\leq \exp\left(-{\dfrac {\delta ^{2}\mu }{4}}\right).$$
\end{thm}

The following bound will be useful for larger deviations.

\begin{thm}\label{binomial-bound}
Let $n\in \mathbb{N},\ p\in [0,1],\ L>0$ an let $X\sim Bin(n,p)$ be a random variable. Then: $$\bP(X\geq L)\leq \binom{n}{L}p^L\leq \left(\frac{enp}{L}\right)^L.$$ 
\end{thm}

Lastly, we will also require Hoeffding's inequality.

\begin{thm}[\textbf{Hoeffding's Inequality}]\label{hoeff}
Let $X_1,X_2,\dots, X_n$ be independent random variables such that $a_i\leq X_i\leq b_i$ for each $i\in[n]$, where $a_i,b_i\in \R$. Then given $t>0$, the random variable $S_n=X_1+\dots+X_n$ satisfies:
$$\bP\left(|S_n-\bE[S_n]|\geq t\right)\leq 2\emph{exp}\left(\dfrac{-2t^2}{\sum_{i\in[n]}(b_i-a_i)^2 }\right).$$
\end{thm}

\section{Degrees and distributions on the continuous cube}
\label{sec: degrees and distributions on the cube}

\subsection{Recasting the problem} 

Given a graph $G$ and a probability vector $\laur{p} = (p_v)_{v \in V(G)} \in [0.1, 0.9]^{V(G)}$ we will write $G(\laur{p})$ to denote the probability space on the set of induced subgraphs of $G$, determined by including each vertex $v \in V(G)$ independently with probability $p_v$. Equivalently, given $S \subset V(G)$, the induced subgraph $G[S]$ is selected with probability $\prod _{v \in S}p_v \prod _{v \in V(G) \setminus S} (1-p_v)$. Abusing notation slightly\footnote{As with the Erd\H{o}s--Renyi random graph $G(n,p)$.}, we will usually write $G(\laur{p})$ to denote a random graph $G[S] \sim G(\laur{p})$.

Throughout the paper, given a vertex $u \in V(G)$, we will we write $\textbf{u} \in \{0,1\}^{V(G)}$ to denote the \emph{neighbourhood vector} of $u$, which is given by: 
    \begin{eqnarray*}
        ({\bf u})_v =
        \begin{cases} 
            1 \quad \mbox{ if } uv \in E(G);\\
            0 \quad \mbox{ otherwise}.
        \end{cases}
    \end{eqnarray*}
\indent Note that, considering the standard inner product on ${\mathbb R}^{V(G)}$, given by ${\bf x}\cdot {\bf y} = \sum _{v\in V(G)}x_vy_v$, this notation leads us to the useful representation:
    \begin{eqnarray}
        \label{eqn: inner product expected degree}
        {\mathbb E}\big [ d_{G(\laur{p})} (u) \big ] = \bf u \cdot \laur{p} .
    \end{eqnarray}
\indent Our first lemma comes to show that two vertices whose expected degrees (under the distribution $G(\laur{p})$) are separated are unlikely to have the same degree in an induced subgraph selected according to $G(\laur{p})$.

\begin{lem}
    \label{lem: probdegcontrol} 
    Let $G$ be a graph and let $u, v$ be distinct vertices in $G$. Suppose that there is a probability vector $\laur{p}\in [0.1,0.9]^V$ such that $\big|\bE[d_{G(\laur{p})}(u)]-\bE[d_{G(\laur{p})}(v)]\big|\geq D \geq 2$. Then:  $$\bP\big(d_{G(\laur{p})}(u)=d_{G(\laur{p})}(v)\big) = O \bigg ( \frac {\sqrt{\log D}}{D} \bigg ).$$ 
\end{lem}

\begin{proof}
Set $W := \diver{u,v}$ and $T = |W|$; by hypothesis $T \geq 2$. Letting $X:= d_{G(\laur{p})}(u)-d_{G(\laur{p})}(v)$, this random variable can be written as $X=\sum_{w\in W}\pm X_w$ where $X_w \sim \mbox{Be}(p_w)$ are independent Bernoulli random variables.  We seek to upper bound $\bP\big(d_{G(\textbf{p})}(u)=d_{G(\textbf{p})}(v)\big) = {\mathbb P}(X = 0)$.

As $\big|\bE[X]\big|\geq D$ by our hypothesis, one gets by Hoeffding's Inequality that:
\begin{align*}
    \bP(X=0) 
        & \leq \bP\big(|X-\bE[X]|\geq D\big)
        \leq 
    2\text{exp}(-D^2/4T),\end{align*} 
since $X$ is a sum of $T$ independent random variables taking values in the interval $[-1,1]$. On the other hand, by Theorem \ref{problitoff} we get $P(X=0)=O\left(T^{-1/2}\right)$. Thus: 
    \begin{eqnarray}
        \label{eqn: joint upper bound on collision prob}
        {\mathbb P}\big ( X = 0 \big ) \leq \min \big \{ O(T^{-1/2}), 2\exp (-D^2 / 4T) \big \}. 
    \end{eqnarray}
\indent The map $x\mapsto 1/\sqrt{x}$ is decreasing on $(0,\infty)$, whereas $x\mapsto \text{exp}(-D^2/4x)$ is increasing, and their intersection point satisfies the equation $\sqrt{x}= \text{exp}(D^2/4x)$, i.e. $D^2= 2x\log x$. This gives $x= \Theta(D^2/\log (D))$ and we get the conclusion by substituting this into \eqref{eqn: joint upper bound on collision prob}. 
\end{proof}

Given a graph $G$, a probability vector $\laur{p} \in [0,1]^{V(G)}$ and $D > 0 $, a set $U \subset V(G)$ is said to be 
\emph{$D$-separated in $G(\laur{p})$} if $|{\mathbb E}[d_{G(\laur{p})}(u) ] - {\mathbb E}[d_{G(\laur{p})}(v) ] | \geq D  $ for all distinct $u, v \in U$.
    
In analogy with $f(G)$, define: 
    \begin{align*}
        {f}_{\laur{p}}(G) 
                := 
        \max \big \{ |U| : U \subset V(G) \mbox { such that } U \mbox{ is  1-separated in } G(\laur{p})\big \}.
    \end{align*}
\indent The next result shows a lower bound for $f(G)$ follows from a lower bound for $f_\laur{p}(G)$.

\begin{thm}
\label{thm: DExpD}
    Given a graph $G$ and a probability vector $\laur{p}\in [0.1,0.9]^{V(G)}$ with $f_\laur{p}(G) \geq 2$, the following relation holds:
        \begin{align*}
            f(G) = \Omega \bigg ( \frac {{f}_\laur{p}(G) }{ \log ^{3/2}\big ( {f}_\laur{p}(G) \big ) } \bigg ).
        \end{align*}
\end{thm}

\begin{proof}
First note that $f(G) \geq 1$ for every non-empty graph $G$, therefore we may assume that $L:= f_\laur{p}(G) \geq C$ for some absolute constant $C$. As above, we will write $G[S]$ to denote a random induced subgraph $G[S] \sim G(\laur{p})$.  Let $U \subset V(G)$ be a $1$-separated set in $G(\laur{p})$ with $U = \{u_1,u_2 \ldots, u_L\}$, so that the vertices are ordered with increasing expected degree in $G(\laur{p})$. It follows that if $j - i \geq 2$ then $D_{i,j}:={\mathbb E}[ d_{G(\laur{p})}(u_j)] - {\mathbb E}[ d_{G(\laur{p})}(u_i)] \geq j-i \geq 2$ and so we can apply Lemma \ref{lem: probdegcontrol} to obtain that:
$$\bP\big(d_{G(\laur{p})}(u_j)=d_{G(\laur{p})}(u_i)\big) \leq \frac {c \sqrt{\log (D_{i,j})}}{D_{i,j}} \leq  \frac {c \sqrt{\log (j-i)}}{j-i},$$
where here $c>0$ is an absolute constant. Here we used that $\sqrt {\log x }/ x$ is decreasing for $ x \geq 2$.

Now let us consider a random graph $H$ on $U_1 = \{u_3, u_6, \ldots, u_{3\lfloor L/3\rfloor }\}$, where we build an edge between two vertices if they have the same degree in $G[S] \sim G(\laur{p})$. The expected number of edges in $H$ is given by:
\begin{eqnarray*}
     \bE[e(H)] = 
\sum _{\{u_{3i}, u_{3j}\}  \subset U_1} \bP\big(d_{G(\laur{p})}(u_{3j})=d_{G(\laur{p})}(u_{3i})\big) 
& \leq & \sum _{\{u_{3i}, u_{3j}\}  \subset U_1} \frac {c\sqrt{\log (3j-3i)}}{3(j-i)} \\ 
& \leq & \dfrac{c L}{9} \sqrt {\log (L)} \cdot \bigg ( \sum _{d = 1}^{L/3}\frac {1}{d} \bigg ) \leq \dfrac{c L}{9} \log ^{3/2}(L ).
\end{eqnarray*}
\indent It follows by Markov that ${\mathbb P} \big (e(H) \leq cL \log ^{3/2}(L)/3 \big ) \geq 2/3$.

On the other hand, we have ${\mathbb E}[ | S \cap U_1| ] \geq |U_1|/10 \geq L/32$ and so by Chernoff's inequality, using that $L \geq C$, we have 
${\mathbb P}( | S \cap U_1| \geq L/64 ) \geq 2/3$.

Combining these two bounds guarantees that there exists an induced subgraph $G[S]$ with the property that the set $U_2 := S \cap U_1$ satisfies $|U_2| \geq L / 64$ and $e(H[U_2]) \leq e(H) \leq cL \log ^{3/2}(L)/3$. From Tur\'an's theorem we see that the subgraph $H[U_2]$ contains an independent set of order $\Omega ( 3|U_2|^2 / cL \log ^{3/2}(L) \big ) = 
\Omega \big (L / \log ^{3/2}(L) \big )$. By definition of $H$ such a set necessarily has distinct degrees in $G[S]$, thus completing the proof.
\end{proof}

\subsection{Moving to distributions}
\label{subsection: move to distributions}

The message we get from Theorem \ref{thm: DExpD} is that a lower bound on $f(G)$ for \emph{any} graph $G$ (up to logarithmic factors) follows from a lower bound on:
    \begin{align*}
        \widetilde {f}(G) := \max _{\laur{p} \in [0.1,0.9]^{V(G)}} f_\laur{p}(G).
    \end{align*}
\indent This second quantity can be perceived as a continuous relaxation of $f(G)\ -$ which trivially corresponds to maximizing over $\{0,1\}^{V(G)}$. However, from our point of view the second solution space is considerably richer, and in particular will allow different behaviours to be blended in a way that is not possible with vectors from the discrete cube; for example, one can take convex combinations of vectors in $[0,1]^{V(G)}$.

Although we would like to lower bound ${\widetilde f}(G)$, this quantity turns out to be just too rigid for certain inductive steps which we want to carry out later\footnote{See comment before Lemma \ref{lem: product-dist-lem} below.}. Instead, we introduce a generalised parameter, defined in terms of probability distributions on $[0.1,0.9]^{V(G)}$, which turns out to be more robust in this respect. 

Let $G$ be a graph and let ${\cal D}$ be a probability distribution on $[0.1, 0.9]^{V(G)}$. Given distinct vertices $u,v \in V(G)$ and a set $S\subset V(G)$, we define:
\begin{align}
    \label{eqn: bad-definition}
\baddd{D}{}{S}{u,v}:=\max_{c\in \mathbb{R}}\ \underset{\laur{p}\sim \mathcal{D}}{\bP}\big(|\bE[d_{G(\laur{p})}^S(u)]-\bE[d_{G(\laur{p})}^S(v)]-c|\leq 1\big).
\end{align}

\indent This quantity can be viewed as a small ball probability $-$ a measurement for two vertices $u,v \in V(G)$ of how likely the expected degrees to $S$  in $G(\laur{p})$ are to differ by an (almost) fixed amount. Given sets $U, S\subset V(G)$, we also set: $$\baddd {\cal D}{}{S}{U}:=\displaystyle\sum_{\{u,v\}\subset U}\baddd {\cal D}{}{S}{u,v}.$$
\indent Given another set $V \subset V(G)$ we can also write: 
$$\baddd {\cal D}{}{S}{U, V }:=\displaystyle\sum_{(u,v)\in U\times V}\baddd {\cal D}{}{S}{u,v}.$$
\indent We will sometimes suppress the superscript when $S = V(G)$, e.g. $\bad U = 
\baddd {\cal D}{}{V(G)}{U}$. Lastly, let us remark that in \eqref{eqn: bad-definition} we do not need $\mathcal{D}$ to be defined on all vertex coordinates of the set $[0.1,0.9]^{V(G)}$; any vertex set $T$ with $S\subseteq T\subseteq V(G)$ is enough so that we can define $\mathcal{D}$ on $[0.1,0.9]^T$, as we can see by looking at the RHS of \eqref{eqn: bad-definition}.

The following lemma shows that a lower bound on $f_\laur{p}(G)$ (and on $f(G)$ by Theorem \ref{thm: DExpD}) follows by finding a large subset $U \subset V(G)$ such that $\bad {U}$ is bounded in terms of $|U|$.

\begin{lem}
    \label{lem: bad-control-implies-distinct-expected-degrees}
    Let $G$ be a graph, let ${\cal D}$ be a probability distribution on $[0.1,0.9]^{V(G)}$ 
    and let $U \subset V(G)$ with $\bad U = \alpha \cdot |U|$. Then there is 
    $\laur{p} \in [0.1,0.9]^{V(G)}$ with
    $f_\laur{p}(G) \geq |U| / (1 + \alpha )$.
\end{lem}

\begin {proof}
    To see this, select $\laur{p} \sim {\cal D}$ and let $Y$ denote the random set:
        $$Y(\laur{p}) :=  \big \{ \{u,v\} \subset U: \big |\bE[d_{G(\laur{p})}(u)]-\bE[d_{G(\laur{p})}(v)] \big | \leq 1 \big \}.$$
    \indent Note that: 
        \begin{eqnarray*}
            \underset{\laur {p} \sim {\cal D}}{\mathbb E} \big [ |Y(\laur{p})| \big ]
                & = &
            \sum _{\{u,v\} \subset U} 
            {\mathbb P} \big ( \big |\bE[d_{G(\laur{p})}(u)]-\bE[d_{G(\laur{p})}(v)] \big | \leq 1 \big )\\ 
               & \leq & 
            \sum _{\{u,v\} \subset U} \bad {u,v} = \bad U = \alpha |U|.
        \end{eqnarray*}
    \indent It follows that there is a choice of $\laur{p} \in [0.1,0.9]^{V(G)}$ such that 
    $|Y(\laur{p})| \leq \alpha |U|$. Viewing the pairs in $Y(\laur{p})$ as the edges of 
    a graph $J$ on the vertex set $U$, again by Tur\'an's theorem we can find an independent set in this graph which has order $|U|/ (1 + \alpha )$. By definition of $J$, this gives a 
    lower bound on $f_\laur{p}(G)$, as required.
\end{proof}

From Theorem \ref{thm: DExpD}, the quantity $f(G)$ is (essentially) lower bounded by $f_\laur{p}(G)$. To close this subsection, and complete the circle, we show that that this also holds in the reverse direction. In particular, up to logarithms the quantities $f(G)$ and ${\widetilde f}(G)$ are of the same order of magnitude.

\begin{lem}
    \label{lem: reverse direction}
    Let $G$ be a graph and let $U\subset S \subset V(G)$ be vertex subsets such that 
    all vertices of $U$ have distinct degrees in $G[S]$. Then there is a 
    distribution ${\cal D}$ on $[0.1, 0.9]^{V(G)}$ such that 
    $\bad U = O \big ( |U| \log |U| \big )$. In particular, 
    there is $\laur{p} \in [0.1,0.9]^{V(G)}$ such that: 
   $$f_{\laur{p}}(G) = \Omega \bigg ( \frac {f(G)}{ \log f(G) } \bigg ).$$
\end{lem}

\begin{proof}
    To see this, let ${\bf s} \in \{0,1\}^{V(G)}$ denote the indicator vector of the 
    set $S$ and let $\bm{1}$ denote the constant $1$ vector. Let $U := \{u_1,u_2,\ldots , u_{|U|}\}$ and assume that 
    $d_{G[S]}(u_i)$ is increasing with $i$, which by \eqref{eqn: inner product expected degree} gives $({\bf u}_j - {\bf u}_i) \cdot {\bf s} \geq j-i$ for all $1 \leq i < j \leq |U|$.
    
    Select 
    $\alpha $ uniformly at random in $[-0.4, 0.4]$ and consider the random vector:
        \begin{align*}
            \laur{p} := \frac {1}{2} \cdot {\bf 1} + \alpha \cdot {\bf s} \in [0.1,0.9]^{V(G)}.
        \end{align*}
    \indent Write ${\cal D}$ for the resulting probability distribution on $[0.1,0.9]^{V(G)}$. 
   Given $\laur{p}$, by \eqref{eqn: inner product expected degree} we get:
        \begin{align*}
            {\mathbb E}[ d_{G(\laur{p})}(u_j)] - {\mathbb E}[ d_{G(\laur{p})}(u_i)] = 
            ({\bf u}_j - {\bf u}_i) \cdot \laur{p} = 
            \alpha \cdot \big ({\bf u}_j - {\bf u}_i \big ) \cdot {\bf s}  
            + c',
        \end{align*}
    for some fixed constant $c'$. As $\big ({\bf u}_j - {\bf u}_i \big ) \cdot {\bf s} \geq 
    j-i$ and $\alpha $ is uniformly chosen from $[-0.4,0.4]$, it follows 
    that ${\mathbb E}[ d_{G(\laur{p})}(u_j)] - {\mathbb E}[ d_{G(\laur{p})}(u_i)]$ is uniformly distributed over 
    an interval of length at least $0.8(j-i)$. By definition 
    \eqref{eqn: bad-definition}, this then gives: 
        \begin{align*}
            \bad {u_i, u_j} \leq \frac{2}{0.8(j-i)} \leq \frac {3}{j-i}.
        \end{align*}       
    \indent It follows that $\bad U = \displaystyle\sum _{1 \leq i<j \leq |U|} \bad {u_i, u_j} \leq 
    \displaystyle\sum _{d=1}^{|U|} \dfrac {3|U|}{d} \leq 6|U| \log |U|$, giving us the first bound. 
    The second then follows immediately from Lemma \ref{lem: bad-control-implies-distinct-expected-degrees}.
\end{proof}

\section{Building distributions for distinct degrees}
    \label{sec: building distributions}

From the previous section, via Theorem \ref{thm: DExpD} and Lemma \ref{lem: bad-control-implies-distinct-expected-degrees}, we know that in order to find many distinct degrees in a graph $G$ it suffices to find a large set $U \subset V(G)$ and a probability distribution ${\cal D}$ such that $\bad U$ is small. In this section we will collect a number of results together, which will be used in combination to exhibit such distributions ${\cal D}$.

From the `iterated' graph examples discussed in Section \ref{section: introduction} we saw that occasionally we must first find distinct degree sets $U_i$ in graphs $G[S_i]$ where $\{S_i\}_i$ are disjoint, and then combine these sets together so that $\bigcup _i U_i$ will have distinct degrees in $G[ \bigcup _i S_i]$. Unfortunately, it is also not hard to see that vertices within $U_i$ can easily agree in degree in the resulting union graph, even if we move from sets $U_i$ to vectors $\laur{p}_i$ as in Section \ref{sec: degrees and distributions on the cube}. 

While working with fixed sets or vectors can cause difficulties, our first lemma shows that the setting of distributions allows more flexibility here: we can combine distributions while maintaining 'bad' control. This flexibility was the key motivation for working in this more generalised setting (indicated in subsection \ref{subsection: move to distributions}).

\begin{lem}
\label{lem: product-dist-lem}
Let $G$ be a graph with vertex partition $V(G)= \bigsqcup_{i=1}^L V_i$ and for each $i\in [L]$ let $\mathcal{D}_i$ be a probability distribution on $[0,1]^{V_i}$. Then taking $\mathcal{D}$ to denote the product distribution $\Pi_{i\in [L]} \mathcal{D}_i$ on $[0,1]^{V(G)}$, for any distinct vertices $u,v\in V(G)$ and any set $S\subset V(G)$, one has: $$\baddd{D}{}{S}{u,v}\leq \min_{i\in [L]} \baddd{D}{i}{S\cap V_i}{u,v}.$$
\end{lem}


\begin{proof}
To see this, take $c\in \mathbb{R}$ and define $X$ to be the random variable: 
$$X(\laur{p}):=\bE[d_{G(\laur{p})}^{S}(u)]-\bE[d_{G(\laur{p})}^{S}(v)]-c.$$ It suffices to prove that $\underset{{\bf p}\sim \mathcal{D}}{\bP}\big(|X|\leq 1\big)\leq \baddd{D}{i}{S\cap V_i}{u,v}$ for all $i\in [L]$ as the result will follow from our definition of $\bad{u,v}$. Let $W_i:=V(G)\setminus V_i$ for each $i\in [L]$. Given $\laur{p}\in [0,1]^{V(G)}$, we denote by $\laur{p}_i$ and $\laur{q}_i$ its projections on $V_i$ and $W_i$, respectively, which are mutually independent.\\
\indent It is easy to see that:
$$X(\laur{p})=\bE[d_{G(\laur{p})}^{S\cap V_i}(u)]-\bE[d_{G(\laur{p})}^{S\cap V_i}(v)]+\bE[d_{G({ \laur{p}})}^{S\cap W_i}(u)]-\bE[d_{G({ \laur{p}})}^{S\cap W_i}(v)]-c.$$
\indent Conditioned on any choice for $\laur{q}_i$, we see that $\bE[d_{G(\laur{p})}^{S\cap W_i}(u)]-\bE[d_{G({ \laur{p}})}^{S\cap W_i}(v)]$ becomes a constant, therefore we obtain that:
$$\underset{{ \bf p}\sim \mathcal{D}}{\bP}\big(|X|\leq 1|\ \laur{q}_i\big)=\underset{{\bf p}_i\sim \mathcal{D}_i}{\bP}\big(|\bE[d_{G(\laur{p}_i)}^{S\cap V_i}(u)]-\bE[d_{G(\laur{p}_i)}^{S\cap V_i}(v)]-c^\prime|\leq 1\big)\leq \baddd{D}{i}{S\cap V_i}{u,v},$$ 
as $\laur{p}_i$ and $\laur{q}_i$ are independent. It follows that $\underset{{ \bf p}\sim \mathcal{D}}{\bP}\big(|X|\leq 1\big)\leq \baddd{D}{i}{S}{u,v}$, as desired. 
\end{proof}

Our second lemma gives a simple situation in which we can obtain `bad' control. Let $G$ be a graph and let $S \subset V(G)$. Let ${\cal U}_S$ denote the \textbf{\emph {uniformly constant distribution}} on $[0.1,0.9]^S$, given by selecting $\alpha \in [0.1,0.9]$ uniformly at random and setting $\laur{p}
= \alpha {\bf 1}_S \in [0.1,0.9]^{S}$.

\begin{lem}
    \label{lem: uniform-distribution-control}
        Let $G$ be a graph, $S \subset V(G)$ and $u, v \in V(G)$ such that 
        $d^S(u) \geq d^S(v) + D$ for some $D >0$. Suppose that 
        ${\cal U}_S$ denotes the uniform constant distribution on $[0.1,0.9]^{S}$, that 
        ${\cal D}'$ denotes a distribution on $[0.1,0.9]^{V(G)\setminus S}$ and that 
        ${\cal D}$ denotes the product distribution 
        ${\cal U}_S \times {\cal D}'$ 
        on $[0.1,0.9]^{V(G)}$. Then $\bad {u,v} \leq 3D^{-1}.$
\end{lem}

\begin{proof}
    First note that by Lemma \ref{lem: product-dist-lem} 
    we have $\bad {u,v} \leq \baddd {\cal U}{S}{S}{u,v}$
    and so it suffices to upper bound this second quantity. 
    
    Taking $c \in {\mathbb R}$, we seek to upper bound the probability that 
    $|{\mathbb E}[ d_{G(\laur{p})}^S(u)] - 
    {\mathbb E}[ d_{G(\laur{p})}^S(v)] - c | \leq 1$, where 
    $\laur{p} \sim {\cal U}_S$. To analyse this, note that: 
    \begin{align*}
    {\mathbb E}[ d_{G(\laur{p})}^S(u)] - 
    {\mathbb E}[ d_{G(\laur{p})}^S(v)] 
    = (\proj {S}{u}{} - \proj {S}{v}{}) \cdot 
    \laur{p}.
    \end{align*}
    \indent Since $\laur{p} = \alpha {\bf 1}_S$ where $\alpha $ is selected uniformly at random from [0.1,0.9], this gives: 
    \begin{align*}
    {\mathbb E}[ d_{G(\laur{p})}^S(u)] - 
    {\mathbb E}[ d_{G(\laur{p})}^S(v)]
    = (\proj {S}{u}{} - \proj {S}{v}{}) \cdot 
    \alpha {\bf 1}_S = \alpha \big (d^S(u) - d^S(v) \big ).
    \end{align*}
    \indent As $\alpha $ varies uniformly over the interval $[0.1,0.9]$ and $d^S(u) - d^S(v) \geq D$ by 
    hypothesis, the quantity
    ${\mathbb E}[ d_{G(\laur{p})}^S(u)] - 
    {\mathbb E}[ d_{G(\laur{p})}^S(v)]$ varies uniformly over an interval of length at least $0.8D$, giving that the probability that 
    $|{\mathbb E}[ d_{G(\laur{p})}^S(u)] - 
    {\mathbb E}[ d_{G(\laur{p})}^S(v)] - c | \leq 1$ is at most $2/(0.8D) \leq 3D^{-1}$.
\end{proof}

We next seek to provide 'bad' control for a set by blending neighbourhood structures together $-$ the idea here has some similarities to that of \cite{JKLY}. Let $G$ be a graph, let $U, S \subset V(G)$, where $U := \{u_1,\ldots, u_k\}$, and let $\beta\in[0,0.4]$. We now let ${\cal B}_\beta(U,S)$ denote the \textbf{\emph {blended probability distribution}} on $[0.1,0.9]^S$, which is defined as follows. First independently select $\alpha _i \in [-\beta,\beta]$ uniformly at random for each $i \in [k]$ and set: 
\begin{align}
    \label{eqn: pre truncated p}
\laur{p}' := \frac {1}{2} \cdot {\bf 1} + \sum _{i\in [k]} \alpha _i \cdot \proj {S}{u}{i} \in \mathbb{R}^{S}.
\end{align}
\indent Having made these choices, the distribution then returns $\laur{p}$, a truncated version of $\laur{p}'$, where:
    \begin{align*}
        \laur{p}_v =
        \begin{cases}
               \laur{p}_v' \quad \mbox{if } \laur{p}_v' \in [0.1,0.9];\\
               0.9 \quad \mbox{if } \laur{p}_v' >0.9;\\
               0.1 \quad \mbox{if } \laur{p}_v' < 0.1.
        \end{cases}
    \end{align*}

Our final lemma in this section provides 'bad' control for blended distributions under certain well-behaved situations. Given $D > 0$, $\gamma \in [0,1]$ and sets $U$ and $S$ as above we say that: 
\begin{itemize}[nosep]
    \item $U$ is \emph{$D$-diverse to $S$} if for all distinct $u,v \in U$ we have $
    |N_G^S(u) \triangle N_G^S(v)| \geq D$.
    \vspace{4pt}
    \item $U$ is  $\gamma $-\emph{balanced to }$S$ if for 
    all $v \in S$ we have 
    $d_G^{U}(v) \leq \gamma |U|$.
\end{itemize}

Let us quickly remark that $U$ is always $1$-\emph{balanced to} $S$.

\begin{lem}
    \label{lem: blended-distribution-control}
        Let $G$ be a graph, $D>0,\ \beta\in(0,0.1),\ \gamma\in(0,1]$ and $U, S \subset V(G)$ such that $U$ is both
        $D$-diverse and $\gamma $-balanced to $S$. Suppose that ${\cal D}'$ denotes a distribution on $[0.1,0.9]^{V(G)\setminus S}$, that
        ${\cal B}_\beta(U,S)$ is the blended probability distribution on $[0.1,0.9]^{S}$ 
         and that 
        ${\cal D}$ is the product distribution ${\cal B}_\beta(U,S) \times {\cal D}'$ 
        on $[0.1,0.9]^{V(G)}$. Then for all $u,v\in U$ one has:
                \begin{equation}\label{eqn:bad-eqn-control}
                    \bad {u,v} \leq \dfrac{2}{\beta D}+D\exp\left(\dfrac{-0.045}{\gamma\beta^2|U|}\right).
                \end{equation}
\end{lem}

\begin{proof}
    Suppose $U=\{u_1,u_2,\ldots,u_{k+1}\}$ and that for each $i\in [k+1]$, given the vector $\laur{p}'$ on $\mathbb{R}^S$ from \eqref{eqn: pre truncated p}, we define the random vector $\laur{q}^i$ on $\mathbb{R}^S$ by $\laur{q}^{i}:=\laur{p}^\prime-\alpha_i\cdot \proj {S}{u}{i}$. The key observation is that $\laur{q}^i$ is independent of $\alpha_{i}$. We will slightly abuse notation by writing $\laur{p}$ for both a vector in $[0.1,0.9]^{V(G)}$ and its projection $\proj{S}{\ul{p}}{}$ onto the coordinate set $S$. We can do this without much of a worry since $\mathcal{D}$ is the product distribution ${\cal B}_\beta(U,S) \times {\cal D}'$. \\
    \indent Fix $c\in \mathbb{R}$, $ i,j\in [k+1]$ and let $E_{i,j}(c)$ denote the event  $\big|\bE[d_{G(\laur{p})}(u_i)]-\bE[d_{G(\laur{p})}(u_j)]-c\big|\leq 1$. According to \eqref{eqn: bad-definition}, to prove the lemma it will suffice to show that: $$\bP(E_{i,j}(c))\leq \dfrac{2}{\beta D}+D\exp\left(\dfrac{-0.045}{\gamma\beta^2|U|}\right).$$ 
    \indent To upper bound $\bP(E_{i,j}(c))$, we might assume that $|N^S(u_i)\setminus N^S(u_j)|\geq |N^S(u_j)\setminus N^S(u_i)|$, so that $|N^S(u_i)\setminus N^S(u_j)|\geq D/2$. Pick a subset $Y_{i,j}\subset N^S(u_i)\setminus N^S(u_j)$ of size $D/2$. We call a vertex $v\in Y_{i,j}$ \emph{naughty} if $\laur{q}_v^i\notin[0.2,0.8]$. We say the set $Y_{i,j}$ is \emph{naughty} if it contains a naughty vertex and we let $F_{i,j}$ denote this event. By the law of total probability we get that: $$\bP\big(E_{i,j}(c)\big)=\bP\big(E_{i,j}(c)|F_{i,j}\big)\cdot \bP(F_{i,j})+\bP\big(E_{i,j}(c)|\ov{F_{i,j}}\big)\cdot \bP(\ov{F_{i,j}})\leq \bP(F_{i,j})+\bP\big(E_{i,j}(c)|\ov{F_{i,j}}\big).$$
\indent Let $v\in S$. Note that $\laur{q}_v^i$ is a sum of $ d_G^{U\setminus\{u_i\}}(v)$ uniform independent random variables, as the coordinates $\textbf{u}_{i_v}$ are non-zero when $v\sim u_i$. Thus by Hoeffding Inequality we get: 
$$\bP\big(\laur{q}_v^i\notin [0.2,0.8]\big)=\bP\big(|\laur{q}_v^i-1/2|>0.3 \big)
\leq 2\exp\left(\dfrac{-2\cdot 0.09}{4\beta^2 
d_G^{U\setminus \{u_i\} }(v) }\right) 
\leq 2\exp\left(\dfrac{-2\cdot 0.09}{4\beta^2 
\gamma |U| }\right),$$
where the final inequality uses that $d_G^{U\setminus \{u_i\} }(v) \leq d_G^U(v) \leq \gamma |U|$ as $U$ is 
$\gamma$-balanced to $S$. By the union bound we get that $\bP(F_{i,j})\leq |Y_{i,j}|\bP\big(\laur{q}_v^i\notin [0.2,0.8]\big)\leq D\exp\big(-0.045(\gamma\beta^2|U|)^{-1}\big)$. 

\indent To compute $\bP\big(E_{i,j}(c)|\ov{F_{i,j}}\big)$ we condition on any choice of $\bm{\alpha}:=(\alpha_{l})_{l\neq i}$ such that $F_{i,j}$ does not hold. Given such a choice, let us first see that $\laur{p}_v^\prime=\laur{q}^i_v+\alpha_i \textbf{u}_{i_v}\in [0.1,0.9]$ for all $v\in Y_{i,j}$ since $|\alpha_i|<0.1$. So none of the $Y_{i,j}$-coordinates of $\laur{p}^\prime$ will get truncated and recall that $\alpha_i$ is independent of $F_{i,j}$. Given a choice of $\bm{\alpha}$, consider now the following expression as a map of $\alpha_i$:  
\begin{align}
\label{eqn: difference of expected degrees analysis}    
f_c(\alpha_i):=\bE[d_{G(\laur{p})}(u_i)]-\bE[d_{G(\laur{p})}(u_j)]-c = (\textbf{u}_i-\textbf{u}_j)\cdot \laur{p} -c .
\end{align}
\indent Having conditioned on $\bm{\alpha}$ above, note that the event $E_{i,j}(c)$ holds only if $f(\alpha_i)$ lies in an interval of length $2$. However, as $\alpha _i$ increases, the contribution from each coordinate of $\laur{p}$ to the inner product on the right hand side of \eqref{eqn: difference of expected degrees analysis} is non-decreasing. Furthermore, the contribution of all of the $Y_{i,j}$-coordinates is exactly $\alpha _i$, since none of these coordinates were truncated from $\laur{p}^\prime$ as we have conditioned on $\overline{F_{i,j}}$. It follows that for $\varepsilon>0$: 
$$f(\alpha_i+\varepsilon)-f(\alpha_i)=\displaystyle\sum_{v\in V(G)}\big ((u_{i})_v-(u_{j})_v \big ) (u_{i})_v\cdot g_{\varepsilon,v}\geq \varepsilon|Y_{i,j}|=\varepsilon D/2,$$ 
where $g_{\varepsilon,v}\geq 0$  for all $v\in V(G)$ and $g_{\varepsilon,v}=\varepsilon $ for $v\in Y_{i,j}$. Therefore, conditioned on $\bm {\alpha }$ as above, if $E_{i,j}(c)$ occurs then $\alpha_i$ lies in an interval of length $4/D$. This happens with probability at most $2\beta^{-1}D^{-1}$ and the result in \eqref{eqn:bad-eqn-control} quickly follows from the law of total probability.
\end{proof}

Before we end this section, we define a simple but convenient distribution. Given a graph $G$ and a set $S\subset V(G)$, let ${\cal T}_S$ denote the \textbf{\emph{trivial $S$-induced probability distribution}}, which is simply the distribution on $[0.1,0.9]^S$ which selects the vector $\laur{p}_0 = \frac {1}{2} \cdot {\bf 1}_S$ with probability $1$.

\section{The Narayanan--Tomon conjecture}
\label{sec: N-T conjecture}

In this section we will prove Theorem \ref{thm: N-T-intro-version}, our approximate version of the Narayanan--Tomon conjecture. From Theorem \ref{thm: DExpD} and Lemma \ref{lem: bad-control-implies-distinct-expected-degrees} it will suffice to prove the following theorem.

\begin{thm}
    \label{thm: NT-distribution-version}
    Let $n \in {\mathbb N}$ and $k \geq 1$ with $n \geq 20000k^2$ and suppose that $G$ be an $n$-vertex graph with $\hom (G) \leq n/25k$. 
    Then there is a set $U \subset V(G)$ and a probability distribution ${\cal D}$ 
    on $[0.1,0.9]^{V(G)}$ such that: 
        \begin{align*}
            |U| 
                = 
            \Omega \bigg ( \frac {k}{\log_2 ^2(k+1)} \bigg ) 
                \quad 
            \mbox{ and } 
                \quad 
            \bad U = O \big (|U| \log |U| \big ). 
        \end{align*}
\end{thm}

\indent The proof will split into two regimes. The first deals with the case where $n = \Omega (k^{5/2})$ and the more difficult second case focuses on the regime $k^2 \leq n = O(k^{5/2})$.

\indent To begin, we first present a quick application of Lemma \ref{lem: blended-distribution-control} that guarantees `bad' control for a set which is $\Omega(k^{3/2})$-diverse.

\begin{lem}\label{lem: div-approx}
Let $G$ be a $n$-vertex graph and suppose that there is a set $U=\{v_1,v_2,\ldots,v_{k+1}\}$ of vertices of $G$ such that $|N(v_i)\triangle N(v_j)|\geq k^{3/2}+k$ for all $i\neq j$ in $[k+1]$. Then there is a probability distribution $\mathcal{D}$ on $[0.1,0.9]^{V(G)}$ such that $\bad{U}\leq 8|U|\log_2|U|$.
\end{lem}

\begin{proof}
Let $S:=V(G)$ so that $|N^S(v_i)\triangle N^S(v_j)| =  |\text{div}(v_i,v_j)|\geq k^{3/2} + k$ for all $i\neq j$ in $[k+1]$. Therefore $U$ is both $(k^{3/2}+k)$-diverse and $1$-balanced to $S$. We let $\mathcal{D} := 
{\cal B}_\beta(U,S)$, where we set $\beta^{-1}:=\sqrt{56(k+1)\log(k+1)}$, and apply Lemma \ref{lem: blended-distribution-control} to obtain for all $i \neq j$ that:
\begin{align*}
\bad{v_i,v_j} &\leq 4\sqrt{14}\cdot \dfrac{\log^{1/2}(k+1)}{k}+(k^{3/2}+k)\cdot \exp\big(-2.52 \log(k+1)\big) \\ 
& \leq k^{-1}\big(4\sqrt{14}\log^{1/2}(k+1)+1\big) .
\end{align*}
\indent As the map $f:[1,\infty)\to \mathbb{R}$ given by $f(x):=16\log_2(x+1)-4\sqrt{14}\log^{1/2}(x+1)-1$ is increasing and positive at $1$, we can now easily deduce that for all $i\neq j$ one has: 
$$\bad{v_i,v_j}\leq k^{-1}\big(4\sqrt{14}\log^{1/2}(k+1)+1\big)\leq 16k^{-1}\log_2(k+1).$$
\indent By summing over all $i\neq j$ in $[k+1]$ we finally deduce that:
\begin{align*}\bad{U}\leq \frac{16\log_2(k+1)}{k}\cdot \dbinom{k+1}{2}=8|U|\log_2|U|,
\end{align*}as required.\end{proof}

\subsection{The case when \texorpdfstring{$n = \Omega (k^{5/2})$}{n is bigger than k to 5/2}} 
\label{subsection: proof in the larger regime}

The next result controls `bad' under the assumption that $G$ has bounded maximum degree.

\begin{lem}\label{lem: bded-deg}
Let $n\in {\mathbb N}, x\in [1,\infty )$ and suppose that $G$ is an $n$-vertex graph with $n\geq 25x\cdot \Delta (G)$ and $\hom (G) \leq n/5x$. Then there is a probability distribution $\mathcal{D}$ on $[0.1,0.9]^{V(G)}$ and a vertex set $U\subset V(G)$ with $|U| =  \lceil x \rceil +1$ such that $\bad{U}\leq |U|$.
\end{lem}

\begin{proof}
Set $k = \lceil x\rceil $, noting that $x \leq k < 2x$. We will first select $U = \{u_1,\ldots, u_{k+1}\}$ step by step over a series of rounds. To do so, we are going to select a `control' set $Y_i$ for each $u_i\in U$, so that $u_i$ is strongly joined to $Y_{i}$, but any $u_j\neq u_i$ in $U$ with $j>i$ is quite weakly joined to $Y_i$. This property will allow us to separate the expected degrees of vertices in $U$ and build the distribution $\mathcal{D}$.

We inductively construct vertex sets $U_i=\{u_1,u_2,\ldots,u_i\}$, $V_i$ and $Y_i$ for $i\in [k]$ so that:
\begin{enumerate}[(i),nosep]
    \item the sets $U_i$, $\{Y_j\}_{j\leq i}$ and $V_i$ are all pairwise disjoint;
    \item $u_{i} \in V_{i-1}$ for all $i\in [2, k]$;
    \item $d^{Y_i}(u_i)=|Y_i|=2k$;
    \item $d^{Y_i}(v) \leq k/2$ for all vertices $v \in V_i$;
    \item $|V_i| \geq n - 5i \Delta (G)$.
\end{enumerate} 

To begin, we set $U_0 = Y_0 = \emptyset $ and $V_0:=V(G)$. Suppose now $i \in [k] $ and that we have found $U_{i-1}, V_{i-1}$ and $\{Y_j\}_{j< i}$ as above and wish to find these sets for $i$. We look at $G_{i}:=G[V_{i-1}]$ and see that it must have a vertex $u_i$ with $d_{G_{i}}(u_i)\geq 2k$; in particular $\Delta(G) = \Delta (G_1) \geq 2k$. If not, then $\Delta(G_{i})\leq 2k-1$ and so by Tur\'an's Theorem we obtain an independent set in the subgraph $G_{i}$ which has size at least $|V_{i-1}|/2k\geq (n- 5(i-1)\Delta (G))/(2k) > (n-5x\Delta (G))/(4x) \geq n/5x$, contradicting the $\text{hom}(G)$ condition from our hypothesis. We now let $U_i:=U_{i-1}\cup \{u_i\}$ and we pick a subset $Y_i\subset N_{G_i}(u_i)$ of size $2k$. We then define the set $Z_i:=\{v\in V_{i-1}:d_{G_i}^{Y_i}(v)\geq k/2\}$ and note that $u_i\in Z_i$. We now let $V_i:=V_{i-1}\setminus(Y_i\cup Z_i)$. Observe that by construction (i)-(iv) hold above, and it just remains to show (v).  

As $|V_i| = |V_{i-1}| - |Y_i \cup Z_i|$, by induction it is enough to show that $|Y_j\cup Z_j|\leq 5 \Delta (G)$. Clearly $|Y_i| \leq 2k$. We bound $|Z_i|$ by double counting the number of edges between $Z_i$ and $Y_i$. From each $z\in Z_i$ there are at least $k/2$ edges going to $Y_i$, hence $e(Y_i,Z_i)\geq k|Z_i|/2$. However, from each $y\in Y_i$ there are at most $\Delta(G)$ edges going to $Z_i$, thus $e(Y_i,Z_i)\leq 2k\Delta(G)$. It follows that $|Z_i|\leq 4 \Delta (G)$ and so $|Z_j\cup Y_j|\leq |Z_j|+2k \leq 4 \Delta (G) + 2k \leq 5 \Delta (G)$, as required.

To complete the proof of the lemma, we set $i := k$ and take $u_{k+1} \in V_{k} \neq \emptyset $. By using (i)-(iv) above we get disjoint sets $U = U_{k} \cup \{u_{k+1}\} = \{u_1,\ldots, u_{k+1}\}$ and $\{Y_j\}_{j\in [k]}$ such that: 
\begin{align}
    \label{eqn: degree-split}
        d_{Y_i}(u_i) \geq d_{Y_i}(u_j) + 3k/2\ \ \ \text{for all }i<j.
\end{align} 
 \indent For each $i\in [k]$ we let ${\cal D}_i$ denote the uniformly constant distribution on $Y_i$, i.e. ${\cal D}_i := {\cal U}_{Y_i}$. Taking $Y_0 := V(G) \setminus (\cup _i Y_i)$, we also let ${\cal D}_0 := {\cal T}_{Y_0}$ denote the trivial distribution induced by the set $Y_0:= V(G) \setminus (\cup _i Y_i)$ (as defined at the end of Section \ref{sec: building distributions}). Lastly, we take ${\cal D}$ to be the product distribution ${\cal D} := \prod _{i\in [0,k]}{\cal D}_i$ on $[0.1,0.9]^{V(G)}$. Note that from Lemma \ref{lem: product-dist-lem}, equation \eqref{eqn: degree-split} and Lemma \ref{lem: uniform-distribution-control}, for all $i<j$ we obtain that:
    \begin{align*}
        \bad {u_i, u_j} 
            \leq 
        \baddd {D}{i}{Y_i}{u_i, u_j}
            \leq 
        \frac {3}{(3k/2)} = \frac{2}{k}.
    \end{align*}
It follows that $\bad {U} \leq \binom {k+1}{2}\big (\frac {2}{k} \big ) = |U|$, as desired.
\end{proof}

We are now in a position to prove Theorem \ref{thm: NT-distribution-version} for $n = \Omega (k^{5/2})$.

\begin{thm}\label{5.2-case}
Let $n \in {\mathbb N}$ and $x \geq 1$ with $n \geq 1000x^{5/2}$. Suppose that $G$ is an $n$-vertex graph with $\hom (G)\leq n/20x$. Then there is a probability distribution $\mathcal{D}$ on $[0.1,0.9]^{V(G)}$ and a vertex set $U\subset V(G)$ with $|U|\geq x +1$ such that $\bad{U}\leq 8|U| \log_2|U|$. 
\end{thm}

\begin{proof}
Let $k := \lceil x\rceil $. We will prove the theorem by induction on $|V(G)|$. To start with, observe that there is nothing to prove when $k\leq 4$ as we can set $\mathcal{D}$ to be any distribution on $[0.1,0.9]^{V(G)}$ and the requirements are trivially satisfied by any $(k +1)$-vertex set $U$, since $\bad{u,v}\leq 1$ for any pair $u,v$ of vertices; such a set $U$ exists as $k +1 \leq 1000x^{3/2}$. In particular, this proves that the theorem holds for the smallest possible case, when $n = 1000$ (where $x$ must equal $1$). We will proceed with the induction step and assume that $k>4$. 

Let $V_0$ be a largest vertex set of $G$ such that $|\diver{u,v}|\geq 2k^{3/2}$ for all $u,v\in V_0$. If $|V_0|\geq k+1$ then we are done by Lemma \ref{lem: div-approx}, otherwise assume that $V_0=\{v_1,v_2,\ldots, v_L\}$ for some $L\leq k$ and now for each $i\in [L]$ define the set $V_i:=\{v\in V(G): |\diver{v,v_i}| < 2k^{3/2}\}$. Due to the maximality of $S_0$ we get $V(G)=\bigcup_{i=1}^L V_i$. 
The proof splits into two cases:\vspace{2mm}

\noindent {\textbf {Case I:}}  Every $j\in [L]$ with $d_G(v_j)\in [{10k^{3/2}},n-1-10k^{3/2}]$ 
satisfies $|V_j| \leq 3k$.\vspace{2mm}

It is easy to see that there are at most $3k^2$ vertices of $G$ that do not lie in a set $V_j$ of size at least $3k$. Moreover, $d_G(v_i)-2k^{3/2}< d_G(v)< d_G(v_i)+2k^{3/2}$ for all $i\in[L]$ and $v\in V_i$. Thus, at least $n - 3k^2$ vertices $v\in V(G)$ have their degree satisfy $d_G(v) \notin [12k^{3/2},n-1-12k^{3/2}]$. Therefore, for all such vertices we have $d_G(v) \leq 12 k^{3/2}$ or $d_G(v) \geq n-1- 12k^{3/2}$. We will assume that at least half of these vertices fulfill the first condition, as otherwise we can follow an identical argument by working with the complement $\overline{G}$ instead. Consequently, we find a set $V\subset V(G)$ of size $|V| \geq (n-3k^2)/2\geq 450x^{5/2}$ with $\Delta(G[V])\leq 12k^{3/2}$. Thus $|V| \geq 25x\Delta (G[V])$  and $\hom (G[V]) \leq \hom (G) \leq (n-6k^2)/10x \leq |V|/5x$, hence we can apply Lemma \ref{lem: bded-deg} to $G[V]$ to obtain a distribution ${\cal D}_1$ on $[0.1,0.9]^V$ and a vertex set $U \subset V$ of size $\lceil x \rceil +1 = k+1$ with $\baddd {D}{1}{}{U} \leq |U|$. We also take ${\cal D}_0 := {\cal T}_{V(G) \setminus V}$ to be the trivial distribution induced by $V(G) \setminus V$, and let ${\cal D} := {\cal D}_0 \times {\cal D}_1$ denote the product distribution  on $[0.1,0.9]^{V(G)}$. By Lemma \ref{lem: product-dist-lem} we obtain $\bad {U} \leq \baddd {D}{1}{V}{U} \leq |U|$, as required.\\

\noindent {\textbf {Case II:}}  There is $j\in [L]$ such that 
$d_G(v_j)\in [10k^{3/2},n-1-10k^{3/2}]$ and $|V_j| \geq 3k$.\vspace{2mm}

We pick a subset $V$ of $V_j$ of size $3k$ such that $v_j\in V$. Next, we set ${X_1}:=N(v_j)\setminus V$ and ${X_2}=:V(G)\setminus(V\cup N(v_j))$. By the choice of $v_j$ note that both $|{X_1}|, |{X_2}| \geq 10k^{3/2}-3k$. Our aim is to show that most vertices in ${X_1}$ have big degree in $V$, whereas most vertices in ${X_2}$ have small degree in $V$. This will allow us to separate the distinct degrees we get in $G[{X_1}]$ from those we get in $G[{X_2}]$. This clustering behaviour is illustrated in the picture below.  

\begin{figure}[htp]
    \centering
    \includegraphics[width=12cm]{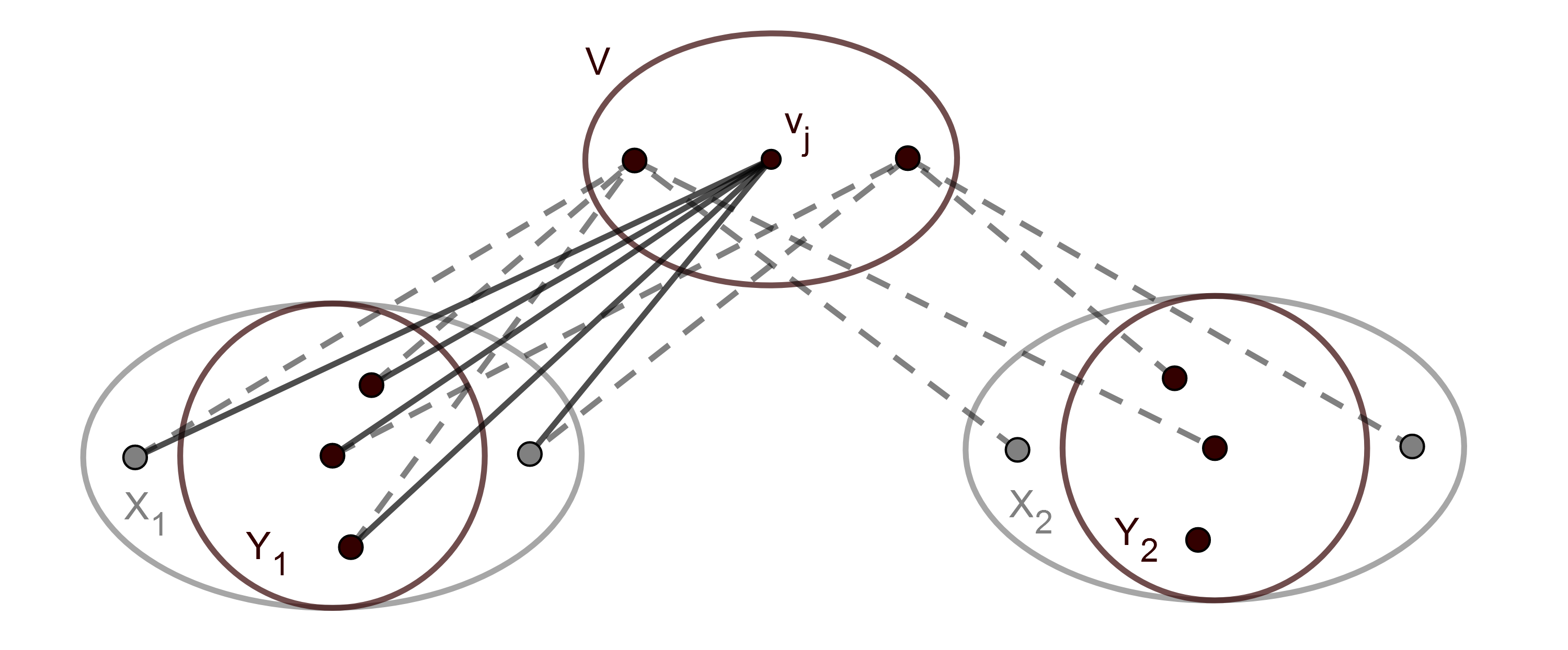}
    \caption{The clustering behaviour of vertices according to their degree in $V$.}
    \label{fig:deg-split}
\end{figure}

To show that this split occurs, we double count the edges in $\ov{G}$ between $X_1$ and $V$. Recall that $X_1\subset N(v_j)$ and for each $v\in V$ we have $|\diver{v,v_j}|\leq 2k^{3/2}$, so each $v\in V$ gives at most $2k^{3/2}$ edges from itself to $X_1$. Hence the number of edges $e_{\hspace{0.3mm}\ov{G}}(X_1,V) \leq (3k)(2k^{3/2}) = 6k^{5/2}$. It follows that there are at most $6k^{3/2}$ vertices of $X_1$ that are connected to less than $2k$ vertices in $V$. Thus, if we let $Y_1:=\{u\in X_1: d_G^V(u)\geq 2k\}$ we see that $t_1:= |Y_1| \geq |X_1| - 6k^{3/2} \geq 4k^{3/2} - 3k > 10.$\\
\indent Similarly, we double count the edges in $G$ between $X_2$ and $V$ to see that there are at most $6k^{5/2}$ of them. It follows that at most $6k^{3/2}$ vertices of $X_2$ that are connected to more than $k$ vertices in $V$. Therefore, if we let $Y_2:=\{u\in {X_2}:d_G^V(u)\leq k\}$, then we can also see that $t_2:=|Y_2| \geq |X_2| - 6k^{3/2} > 4k^{3/2}-3k > 10 $. Recalling that $V(G) = V \cup X_1 \cup {X_2}$ is a partition, this shows that $Z:=V(G)\setminus(Y_1 \cup Y_2)$ satisfies $|Z|\leq 2\cdot 6k^{3/2}+3k\leq 15k^{3/2}$. 

\indent To complete the proof, we apply the induction hypothesis to both $Y_1$ and $Y_2$. 
For $i\in \{1,2\}$ let $x_i:=x(t_i/n) \leq x$. This gives $\hom (G[Y_i]) \leq \hom (G) \leq n/20x=t_i/20x_i$. Furthermore: 
\begin{align*}
    \frac{t_i}{x_i^{5/2}}=\frac {1}{x_i^{3/2}}\cdot \frac {t_i}{x_i}=\frac {1}{x_i^{3/2}}\cdot  \frac {n}{x} \geq \frac {n}{x^{5/2}} \geq 1000.
\end{align*}
\indent Thus, for $i\in\{1,2\}$, provided $x_i \geq 1$ holds, we can apply the induction hypothesis to $G[Y_i]$ to find a probability distribution $\mathcal{D}_i$ on $[0.1,0.9]^{Y_i}$ and a set $U_i\subset Y_i$ satisfying
$|U_i|\geq x_i+1$ and:
   \begin{align}
        \label{eqn: inductive-control}
    \baddd{D}{i}{}{U_i}\leq |U_i| \cdot f(|U_i|),\ \ \ \text{where }f(y):=8\log_2y.
   \end{align}
   \indent Also note that if instead $x_i <1$ above, then as 
   $|Y_i| = t_i \geq 10$, we can take any set $U_i \subset Y_i$ of order $ \lceil x_i \rceil +1 = 2$ and any 
   distribution ${\cal D}_i$ on $[0.1,0.9]^{Y_i}$, so 
   \eqref{eqn: inductive-control} holds in all cases.\\
   \indent We can also assume that $\max \{|U_1|, |U_2| \} < k+1$, as otherwise taking $U$ to simply be one of these sets proves the theorem. 
   
\indent We will also let ${\cal D}_0 := {\cal U}_V$ denote the uniformly constant distribution on $[0.1,0.9]^V$ and let ${\cal D}_3 := {\cal T}_{Z}$ denote the trivial $Z$-induced distribution. We now set $U:=U_1\cup U_2$ and let ${\cal D}$ denote the product distribution $\prod^3_{i=0} {\cal D}_i$ on $[0.1,0.9]^{V} \times \prod _{i\in [2]} [0.1,0.9]^{Y_i}\times [0.1,0.9]^Z = [0.1,0.9]^{V(G)}$. 

Note that $d^V_G(u) \geq 2k \geq d_G^V(v) + k$ for all $u \in Y_1$ and $v \in Y_2$, by definition of $Y_1$ and $Y_2$. It then follows from Lemma \ref{lem: uniform-distribution-control} that for all such vertices we have:
    \begin{align}
        \label{eqn: split-control-to-V}
        \baddd {D}{0}{V}{u,v} \leq \frac {3}{k}.
    \end{align}
\indent As $n \geq 1000x^{5/2}$ and $|Z| \leq 15k^{3/2}$, we can now lower bound the size of $U$:
\begin{align*}
    |U|
        =
    |U_1|+|U_2|
    &\geq (x_1 + 1)+(x_2+1) 
    \geq x\bigg ( \frac {t_1}{n} \bigg )+x \bigg ( \frac {t_2}{n} \bigg ) + 2 \\
    &\geq \frac{(n-|Z|)x}{n} + 2 \geq x- \frac{15x\cdot k^{3/2}}{n} + 2 \geq x+1,
\end{align*}
which gives $|U| \geq x+1$. Finally, we are able to estimate $\bad{U}$ as follows: 
\begin{align*}
    \bad{U}
       & = \sum_{\{u,v\}\subset U_1}\bad{u,v}+\sum_{\{u,v\}\subset U_2}\bad{u,v}+\sum_{(u,v)\in U_1 \times U_2}\bad{u,v}\\
    &=\bad{U_1}+\bad{U_2}+|U_1||U_2|\cdot \max _{(u,v)\in U_1 \times U_2} \big \{\bad{u,v} \big \}\\
    &\leq \baddd{D}{1}{}{U_1}+\baddd{D}{2}{}{U_2}+|U_1||U_2| \cdot \max _{(u,v)\in U_1 \times U_2} \big \{\baddd {D}{0}{V}{u,v} \big \}\\
    &\leq 
    |U_1|\cdot f(|U_1|) + |U_2| \cdot f(|U_2|) + \frac {3}{k} \cdot |U_1||U_2| \leq |U|\cdot f(|U|).
\end{align*}
\indent The final three inequalities here respectively follow from Lemma \ref{lem: product-dist-lem}, then from 
\eqref{eqn: inductive-control} and \eqref{eqn: split-control-to-V}, and lastly from $\max \{|U_1|, |U_2|\} < k+1$ and Lemma \ref{log2ineq}. This completes the proof. 
\end{proof}

\subsection{The case when \texorpdfstring{$n = O(k^{5/2})$}{n is smaller than k to 5/2}}
\label{subsection: proof in the entire regime}

Before we move to the case when $n=O(k^{5/2})$, we present two results which will allow us to move to a large induced subgraph, which is reasonably regular. Comparable results, with a different range of parameters, were proved by Alon, Krivelevich and Sudakov in \cite{alon2008large} (Section 2). The next lemmas follow their approach. We first introduce the following notion. 

\begin{dfn}
For every $n$-vertex graph $G$, its average degree, denoted by $\ov{d}(G)$, is given by the formula $\ov{d}(G):= n^{-1}\sum_{v\in V(G)}d_G(v)$.
\end{dfn}

\begin{lem}
Every $n$-vertex graph $G$ contains an induced subgraph $H$ of order at least $n/3$ such that $\Delta(H)\leq 2\log_2n\cdot \ov{d}(H)$.
\end{lem}

\begin{proof}
We set $G_0:=G$ and for $i=0$ to $i=\log_2n$ we repeat the following algorithm:
first set $n_i:=|V(G_i)|,\ \Delta_i:=\Delta(G_i)$ and $d_i:=\ov{d}(G_i)$. Then, if $\Delta_i\leq 2d_i\log_2n$ we simply stop the process. Otherwise we repeatedly delete from $G_i$ all vertices of degree at least $d_i \log_2n$ to create a new graph $G_{i+1}$. Let $H$ be the graph we obtain after we complete the algorithm.\\
\indent Observe that at $i^{\text{th}}$ iteration we delete at most $e(G_i)/(d_i \log_2n)=n_i/(2\log_2n)$ vertices, therefore $n_{i+1}\geq n_i(1-(2\log_2n)^{-1})$. It follows that $|V(H)|\geq n\cdot (1-(2\log_2n)^{-1})^{\log_2n}$. As $1-x\geq e^{-2x}$ for $0<x\leq 1/2$, we deduce that $|V(H)|\geq n/e>n/3$.\\ 
\indent If $H$ was created because at some point $\Delta_i\leq 2d_i\log_2n$ then we are done. Otherwise $H$ was obtained after $\log_2n$ iterations and at each step $i$ we have $\Delta_{i+1}\leq d_i\log_2n$ and $2d_i\log_2n \leq \Delta_i$. Thus we see that $\Delta_{i+1}\leq \Delta_i/2$. It follows inductively that $\Delta(H)\leq \Delta(G)\cdot 2^{-\log_2n}<n\cdot n^{-1}=1$. We then get that $\Delta(H)=\ov{d}(H)=0$, which also ends the solution. 
\end{proof}

\begin{lem}\label{lem-reg}
Every $n$-vertex graph $G$ contains an induced subgraph $H$ that is of order at least $n/30\log_2n$ with $\Delta(H)\leq 5\log_2n\cdot \delta(H)$. 
\end{lem}

\begin{proof}
By the previous lemma we can find an induced subgraph $G_0$ of $G$ of order $m\geq n/3$ such that $\Delta(G_0)\leq 2\log_2n\cdot \ov{d}(G_0)$. We now perform the following algorithm: starting with $i=0$, let $d_i:=\ov{d}(G_i)$ and delete a vertex $v$ of $G_i$ if $5d_{G_i}(v)<2d_i$. Let now $G_{i+1}$ be the resulting graph and increment $i$. Note that at each step we remove from $G_i$ at most $2d_i/5$ edges, which implies that $d_{i+1}|G_{i+1}|\geq d_i|G_i|-4d_i/5>d_i(|G_i|-1)$, thus $(d_i)_{i\geq 0}$ is an increasing sequence. Therefore we stop before deleting all the vertices and we let $H$ be the resulting graph. \\
\indent We can now observe that $\Delta(H)\leq \Delta(G_0)$ and $\delta(H)\geq 2d_0/5$, which immediately implies that $\Delta(H)\leq \Delta(G_0)\leq 2d_0\log_2n\leq 5\log_2n\cdot \delta(H)$. We finally have to lower bound the number $t$ of vertices that are left in $H$. When we created $H$ from $G_0$ we deleted less than $2(m-t)d_0/5$ edges, hence $2td_0\log_2n \geq t\Delta(H)\geq t\ov{d}(H)\geq md_0-4(m-t)d_0/5$. By rearranging the last inequality we obtain $t\geq m/(10\log _2n) \geq  n/(30\log_2n)$ and so $H$ is the required induced subgraph.
\end{proof}

We are interested in finding sets that have many diverse pairs of vertices as they will give us the freedom required to select vertices with distinct degrees. We thus make the following:

\begin{dfn}
Given a graph $G$ and $\varepsilon>0$, its \emph{diversity graph} $J_{\varepsilon}(G)$ is the graph on $V(G)$ with an edge between vertices $u$ and $v$ if $|N_G(u)\triangle N_G(v)|\leq \varepsilon\min \{|N_G(u)|,|N_G(v)|\}$. 
\end{dfn} 

The following theorem is the main component of our proof in this case. We note that our earlier results from Subsection \ref{subsection: proof in the larger regime} will be crucial here.

\begin{thm}\label{longest-ever}
Let $G$ be a $n$-vertex graph and let $k\in \mathbb{N}$ be such that $1000k^{5/2}\geq n\geq 8000k^2$, $\text{hom}(G)\leq n/12k$ and $\Delta(G)\leq 4nk^{-1/3}$. There is a probability distribution $\mathcal{D}$ on $[0.1,0.9]^{V(G)}$ and a vertex set $U\subset V(G)$ with $|U|= \Omega\big(k\log_2^{-2}(k+1)\big)$ and $\bad{U}\leq 8|U|\log_2 |U|$.
\end{thm}

\begin{proof}
We first note that if $k$ is small then there is nothing to prove, so we can assume $k>2^{40}$. Moreover, together with the hypothesis this gives: 
\begin{align}\label{eqn:log-corres}
    20+4\log_2k\leq 2\log_2n\leq 20+5\log_2k \leq (5.5) \log _2 k \leq k / 100.
\end{align}
\indent   Next, we apply Lemma $\ref{lem-reg}$ to find an induced subgraph $H$ of $G$ of order $m \geq n\log_2^{-1}n/30$ with $\Delta(H)\leq 5\log_2n\cdot \delta(H)$. From now on we will only work with this subgraph $H$. Notice that $\Delta(H)\geq k\log_2^{-1}n/10$, as otherwise by Tur\'an's Theorem, combined with \eqref{eqn:log-corres}, we find an independent set in $H$ (and so in $G$) of order at least $m/(\Delta(H)+1)\geq n(3k+30\log_2n)^{-1}>n/4k$, contradicting the hypothesis.

Take $J$ to denote the diversity graph 
$J := J_{\varepsilon }(H)$, where $\varepsilon = 1/48$. We then set:
    \begin{align*}
        S_1:= \bigg \{v\in V(H):d_{J}(v)\leq \frac {m}{600k} \bigg \} \quad \mbox{and} \quad  S_2:=V(H)\setminus S_1.
    \end{align*}
\indent Our proof will split according to the sizes of $S_1$ and $S_2$. \vspace{2mm}

\noindent \ul{\textbf{CASE 1:}} $|S_1|\geq m/2$. 

\indent We will show that in this scenario we can take the desired set $U \subset S_1$. We select a set $W \subset S_1$ by including every element of $S_1$ independently with probability $p:=8k/|S_1|$.\\
\indent We now claim that each of the following events holds with probability at least $3/4$: 
    \begin{enumerate}[(i),nosep]
        \item \ \ $|W| \geq 4k$ ;
        \item \ \ $e(J[W]) \leq k$ ;
        \item \ \ $d^W_H(v) \leq 2 \log_2 n \cdot m_{\Delta }$ for all $v \in V(H)$, where 
        $m_{\Delta } := \max\{1,240\cdot\Delta(H)\cdot k/n\}$. 
    \end{enumerate}
    
\indent To prove the claim for (i)$-$(iii) above, let us first denote by $\mathcal{A}_i,\mathcal{A}_{ii}$ and $\mathcal{A}_{iii}$ the events that $|W|\leq 4k$, $e(J[W])\leq 2k$ and $d^W_H(v) \geq 2 \log n \cdot m_{\Delta }$, respectively.\\
\indent Starting with (i), note that $|W|\sim Bin(|S_1|,p)$ with $\bE[|U|]=p|S_1| = 8k$, therefore by Chernoff's Inequality
we get $\bP ({\cal A}_i) = \bP \big (|W|\leq 4k \big )\leq \exp(-k) < 1/4$, proving it for (i).\\
\indent  For (ii) observe that $\bE[e(J[U])]\leq p^2e(J[S_1]) \leq p^2|S_1|(m/600k) \leq 
      64km/600|S_1| \leq k/4$. From Markov's inequality we get ${\mathbb P}({\cal A}_{ii}) = {\mathbb P}(e_J[U] \geq k) \leq 1/4$, which gives us (ii).\\
\indent Lastly, for (iii) take $v\in V(H)$ and let $n_v:= d^{S_1}(v)$. Then note that $d_H^W(v)\sim Bin(n_v,p)$. Now Theorem \ref{binomial-bound} gives $\bP\big(d_H^W(v)\geq 2m_v\log_2n\big)\leq 2^{-2\log_2n}=n^{-2}$, where $m_v:=\max\{1,240n_vk/n\}$. As $m_\Delta\geq m_v$ for all $v\in V(H)$, the union bound gives ${\mathbb P}({\cal A}_{iii}) \leq  n^{-1} < 1/4$.

Combining the above bounds gives us $\bP(\mathcal{A}_{i})+\bP(\mathcal{A}_{ii})+\bP(\mathcal{A}_{iii}) \leq 3/4$. Therefore, by using the union bound we can choose a set $W\subset S_1$ that satisfies all the conditions in (i)$-$(iii).

To continue the proof in this case, note that by (i) and (ii) we can apply Tur\'an's theorem to $J[W]$ to find an independent set $U_0\subset W$ with $|U_0| = 2k+1$. However, this means that $U_0$ is $\big ( \delta (H)/48 \big )$-diverse to $V(H)$. By (iii) the set $U_0$ is $\gamma $-balanced to $V(H)$, where $\gamma := \log_2 n\cdot m_{\Delta }/k$. Letting ${\cal D}:= {\cal B}_{\beta }(U_0,V(H))$ denote the blended probability distribution on $[0.1,0.9]^{V(H)}$, by applying Lemma \ref{lem: blended-distribution-control} with $\beta^{-1}:=10\log_2n\sqrt{m_\Delta}$ we obtain that for all distinct $u,v \in U_0$:
$$\bad {u,v} \leq \frac{960\log_2n\sqrt{m_\Delta}}{\delta(H)}+\frac{\delta (H)}{48}\exp\left(\frac{-4.5\cdot m_\Delta\log_2^2n}{2m_\Delta\log_2n}\right).$$
\indent By noting that $\Delta:=\Delta(H)\leq 5\log_2n\cdot \delta(H)$, this can be further reduced to:
$$\bad {u,v} \leq \frac{12\cdot(20\log_2n)^2\cdot\sqrt{m_\Delta}}{\Delta}+\frac{\delta(H)}{48n^2}.$$
\indent Our next claim is that $\Delta^{-1}\sqrt{m_\Delta}< 28kn^{-1}\log_2k$. Indeed, on the one hand, when $m_\Delta=1$ then $\Delta^{-1}\sqrt{m_\Delta}\leq \Delta^{-1}\leq 10kn^{-1}\log_2n<28kn^{-1}\log_2k$ by \eqref{eqn:log-corres}, as required. On the other hand, $m_{\Delta } \geq 1$ implies $\Delta^{-1}\leq 240kn^{-1}$ and so $\Delta^{-1}\sqrt{m_\Delta}\leq \sqrt{240k/(n\Delta)}\leq 240kn^{-1}<28kn^{-1}\log_2k$, which proves the claim.\\
\indent Recall that  $\log_2n \leq 3 \log_2k$ by \eqref{eqn:log-corres} and that $n\geq 8000k^2$ and $\delta(H)<n$. Therefore, we can deduce that for all distinct $u,v \in U_0$ we have:
\begin{align*}
    \bad{u,v}\leq 12 \cdot (60 \log_2 k)^2 \cdot \frac {28k \log_2 k}{n } + \frac {1}{10^5k^2} \leq \frac {10^3 (\log_2 k)^3}{k}.
\end{align*}
\indent To complete the proof in this case, we choose a subset $U\subset U_0$ of size $10^{-3}k\log_2^{-2}k \geq k^{1/4}$. It follows that $\bad{u,v}\leq 16|U|^{-1}\log_2|U|\ \ \text{for all }u,v\in U.$\\
\indent By summing over all pairs of distinct vertices in $U$, it immediately follows, as required, that: 
$$\bad {U} \leq \frac{16\log_2|U|}{|U|} \cdot \dbinom{|U|}{2}=8|U|\log_2|U|.$$

\noindent \ul{\textbf{CASE 2:}} $|S_2|\geq m/2$.

\indent Our first step here is to find a set $W \subset S_2$ and for each vertex $w\in W$ two sets $S_w, T_w \subset V(H)$ with the following properties: 
\begin{enumerate}[(i),nosep]
    \item $|W| \geq |S_2|/16\Delta(H)$;
    \item $S_w \subset N_H(w)$ and $|S_w| \geq |N_H(w)|/2$ for each $w \in W$;
    \item $S_w \cap  N_H(w') = \emptyset $ for all distinct $w,w' \in W$;
    \item $T_w \subset N_J(w)$ with $|T_w| = t := 2^{-19}\cdot 9k\log_2^{-2}k$ for all $w \in W$. 
\end{enumerate}

    With these sets in hand, our set $U$ will (roughly) be of the form $U = \bigcup _{w \in W} U_w$, where each $U_w$ is a set produced by applying Theorem \ref{5.2-case} to $S_w\subset N_H(w)$, while the sets $T_w$ will be used to establish `bad' control between vertices in distinct $U_w$.

As the following diagram suggests, our partition is guided by the neighbourhoods of vertices in the set  $W = \{w_i\}_i$. The high `$J$-degree' of vertices in $S_2$ guarantees a strong clustering behaviour, so that each vertex $w_i$ has a large set $T_{w_i}$ of vertices which behave very similarly. These sets can be used to obtain `bad' control between vertices in distinct $S_{w_i}$.

\begin{figure}[htp]
    \centering
    \includegraphics[width=12cm]{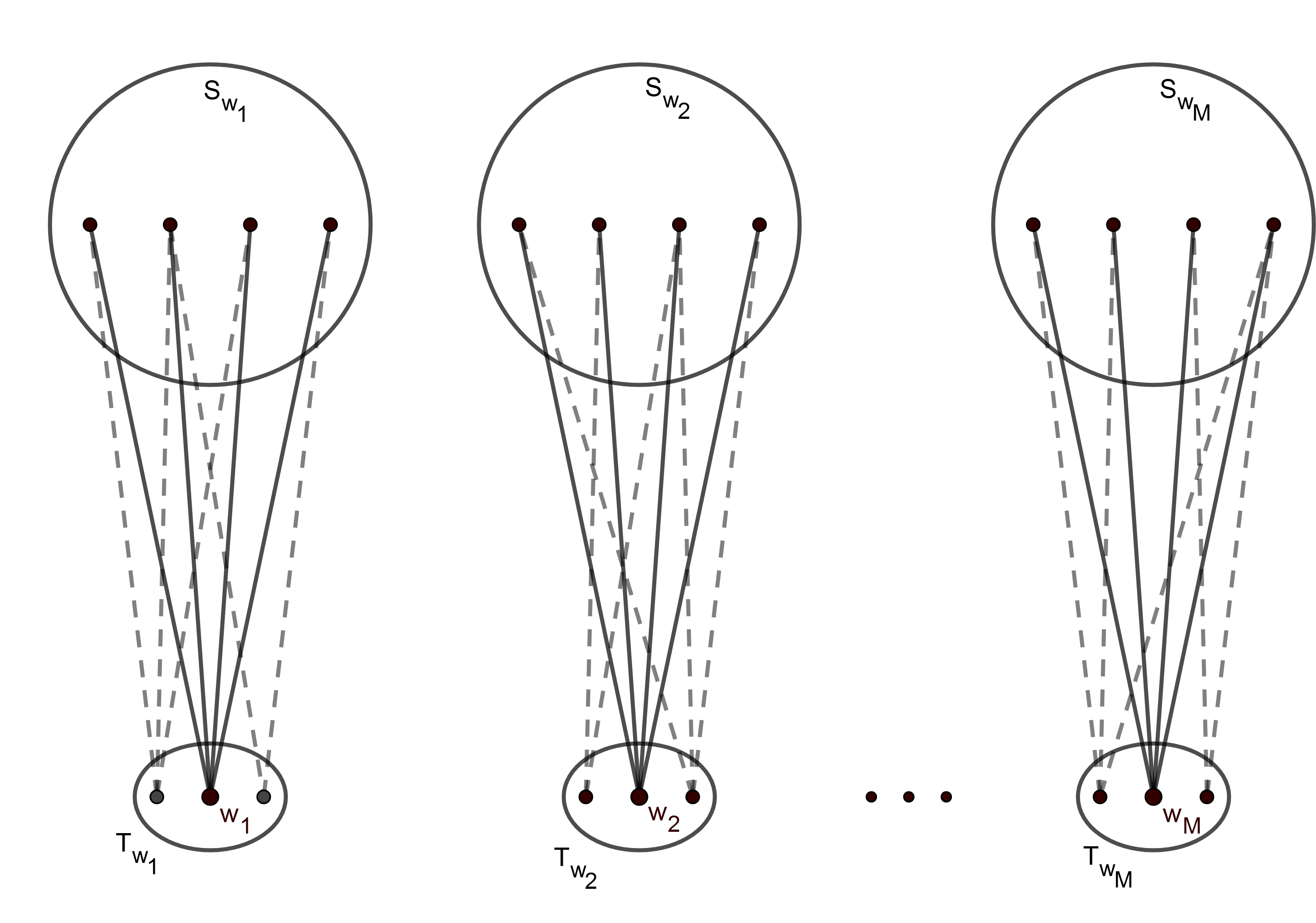}
    \caption{The clusters formed around each $w\in W$.}
    \label{fig:cluster}
\end{figure}

We now proceed with the details. To begin, select a set $W_0 \subset S_2$ by including each element independently with probability $p:= 1/8\Delta $, where $\Delta: = \Delta (H)$. For each $w \in W_0$ we set: 
    \begin{align*}
        S_w:= \big \{v \in V(H): N_H(v) \cap W_0 = \{w\} \big \}.
    \end{align*}
\indent We then let $W \subset W_0$ be the set $W := \{w \in W_0: |S_w| \geq |N_H(w)| / 2\}$. Lastly, each $w \in W$ is also an element of $S_2$, by definition, so we have $d_J(w) \geq m/600k > t$. We take $T_w$ to be an arbitrary subset of size $t$ from $N_J(w)$. 

Having specified the sets, it remains to show that with positive probability properties (i)$-$(iv) hold for our choices. To see this, note that (ii) holds by definition of $S_w$ and $W$. Property (iii) also always holds as if $v \in S_w\cap N_H(w')$ then $v \in N_H(w) \cap N_H(w')$ and $\{w, w'\}  \subset N_H(v) \cap W_0$, which by definition of $S_w$ implies $w = w'$. Lastly (iv) immediately holds by construction.\\
\indent It only remains to prove that (i) holds with positive probability. To see this, note that given $w\in S_2$ and $v\in N(w)$ we have:         
        $$\bP \big (v\in S_w \big |w \in W_0 \big )
            =
        (1-p)^{d^{S_2}_H(v)-1}>(1-p)^\Delta\geq e^{-2p\Delta}
            =
        e^{-1/4}.$$ 
\indent Thus $\bP(v\notin S_w|w \in W_0)\leq 1-e^{-1/4}\leq 1/4$ and so $\bE \big [|N(w)\setminus S_w| \big | w \in W_0 \big ]\leq |N(w)|/4$. It follows from Markov's inequality that: 
        $${\mathbb P}(w \notin W| w \in W_0)  = 
        \bP \big (|N(w)\setminus S_w|\geq |N(w)|/2 \big | w \in W_0 \big )\leq 1/2.$$
\indent We can now further deduce that:
$${\mathbb E}\big [|W_0\setminus W| \big ] = \sum _{w \in S_2} {\mathbb P}(w \notin W| w \in W_0) \cdot {\mathbb P}(w\in W_0) \leq {\mathbb E}\big [|W_0| \big ] / 2 = |S_2|p/2,$$ 
since ${\mathbb E}[|W_0|] = |S_2|p$ because $|W_0| \sim \mbox{Bin}(|S_2|, p)$. It follows that ${\mathbb E}[|W|] = { \mathbb E }[|W_0| - |W_0\setminus W|] \geq |S_2|p/2 = |S_2|/16\Delta $. Thus we can fix a choice of $W$ so that (i), and hence (i)$-$(iv), are satisfied.

Our current aim is to find distinct expected degrees in each subgraph $G[S_w]$ with $w\in W$ by appealing to Theorem \ref{5.2-case} and to use the control sets $\{T_w\}_{w \in W}$ that ensure we can control the degrees between the different sets, so that we can find our required set $U$ in $\bigcup_{w\in W} S_w$.\\
\indent To proceed with this, first observe that the sets $\{S_w\}_{w \in W}$ are pairwise disjoint, since for distinct $w,w'\in W$ we have $S_w \cap S_{w'} \subset S_w \cap N_H(w') = \emptyset $ by (ii) and (iii).\\
\indent Next, notice that the sets $\{T_w\}_{w \in W}$ are also pairwise disjoint. Indeed, suppose there is some $v\in T_{w_1}\cap T_{w_2}$ for some distinct $w_1,w_2 \in W$ and assume $|N(w_1)|\leq |N(w_2)|$. Let $S_v:=S_{w_2}\cap N(v)$ and $\ov{S_v}:=S_{w_2}\setminus N(v)$. As $v\sim w_1$ in $J$, $N_H(w_1)\cap S_{w_2} = \emptyset $ and $S_{w_2} \subset N_H(w_2)$, we immediately see that $|S_v|\leq \varepsilon|N(w_1)|\leq \varepsilon|N(w_2)|$. However $v\sim w_2$ in $J$, thus $|\ov{S_v}|\leq \varepsilon|N(w_2)|$. It follows that $|N(w_2)|/2\leq |S_{w_2}|=|S_v|+|\ov{S_v}|\leq 2\varepsilon|N(w_2)| = |N(w_2)| / 24$, which is a contradiction. Therefore $T_{w_1}\cap T_{w_2}=\emptyset$ for any $w_1\neq w_2$ in $W$.

\indent We want to ensure that vertices of $S_w$ have high degree in $T_w$, whereas their degree in $T_{w^\prime}$ with $w^\prime\neq w$ is low. Given $w\in W$ we define: 
    \begin{align*}
        R_w:= \big \{v\in \bigcup_{w^\prime\neq w} S_{w^\prime}:d_H^{T_w}(v)\geq t/3 \big \}, 
            \qquad \mbox{and} \qquad 
        L_w:= \big \{v\in S_w:d_H^{T_w}(v)\leq 2t/3 \big \}.
    \end{align*} 
\indent If we count the non-edges between $S_w$ and $T_w$ we see there are at least $(t/3)|L_w|$ of them, whereas their number is at most $t(\varepsilon |N(w)|)$ since each vertex of $T_w$ is connected to $w$ in $J$. It follows that $|L_w|\leq 3 \varepsilon |N(w)|\leq 3 \varepsilon (2|S_w|) \leq |S_w|/8$, using (ii) above and that $\varepsilon = 1/48$. Similarly, by double counting the edges between $\bigcup_{w^\prime\neq w} S_{w^\prime}$ and $T_w$ we obtain $|R_w|\leq |S_w|/8$. 

We now set $S^\prime _w:= S_w\setminus\bigcup_{v\in W}(L_v\cup R_v\cup T_v) \subset S_w$ for each $w \in W$. Discarding elements if necessary, we may assume that $|S^{\prime }_w| > 1$ for all $w \in W$. As the sets $\{S_w\}_{w \in W}$ are pairwise disjoint, this also holds for $\{S_w^{\prime}\}_{w \in W}$. From our bounds above we find that: 
    \begin{align*}
        \big | \bigsqcup _{w \in W} S^{\prime }_w \big | 
            \geq 
        \big | \bigsqcup _{w \in W} S_w \big | - \big |\bigcup_{v\in W}(L_v\cup R_v\cup T_v)\big | - \big|W\big|
            &\geq 
        \sum _{w \in W} \big ( |S_w| - 
        |L_w| - |R_w| - |T_w |-1 \big )\\
            & \geq 
      \sum _{w \in W} \big ( |S_w| - |S_w|/4 - t-1 \big )\\ 
          & \geq 
        \sum _{w \in W} \frac{|S_w|}{2} \geq 
        |W|\frac{\delta (G)}{4}.
    \end{align*}
\indent The second inequality here comes from $|L_w|, |R_w| \leq |S_w|/8$, whereas the third one uses that: $$1+t:=1+\frac{9k}{2^{19}\log_2^2k}\overset{\eqref{eqn:log-corres}}{\leq} \frac{k}{400\log_2^2n}\leq \frac{\Delta(H)}{40\log_2 n}\leq \frac{\delta (H)}{8} \leq \frac{|S_w|}{4}.$$
\indent The final inequality above comes from (ii). Continuing with the previous expression, using that $\delta (H) \geq \Delta /(5\log_2 n)$ and that, by property (i), $|W| \geq |S_2|/16\Delta \geq m/32\Delta$, we obtain: 
    \begin{align}
        \label{eqn: lower bound on union of S}
        \sum _{w \in W} \big | S^{\prime }_w \big |
            \geq 
        \frac{|W|\delta (G)}{4}
            \geq 
        \bigg ( \frac {|S_2|}{16\Delta } \bigg ) 
        \bigg ( \frac {\Delta }{20 \log_2 n} \bigg )
            = 
        \frac {m}{640 \log_2 n} 
            \geq 
        \frac {n}{2^{15} \log_2 ^{2}n}.
    \end{align}
\indent We are now in good position to find the desired set $U$. To do this, our main aim is to apply Theorem \ref{5.2-case} to each graph $G[S_w']$. With this in mind, for each $w\in W$ let $k_w:=|S^\prime_w|k/n$, and note that $|S_w^\prime|/k_w=n/k$.  Also recall that $|S^{\prime }_w| \leq |N_H(w)| \leq \Delta\leq 4nk^{- 1/3}$, hence: 
\begin{align*}
    \frac{|S^\prime_w|}{k_w^{5/2}}
        =
    \frac {|S^{\prime }_w|}{(|S_w^\prime|k /n)^{5/2}} 
        =
    \frac {n^{5/2}}{k^{5/2}|S^\prime_w|^{3/2}} 
        \geq 
    \frac {n^{5/2}}{k^{5/2}\cdot\Delta ^{3/2}} 
        \geq 
   \frac {n^{5/2}}{k^{5/2} (4nk^{-1/3})^{3/2}} = \frac {n}{8k^2} \geq 1000.
\end{align*}
\indent We also have $\hom (G[S^{\prime}_w]) \leq \hom (G) 
\leq n / 12k = |S_w^{\prime }|/12k_w$. Thus for each $w \in W$, provided $k_w \geq 1$, we can apply Theorem \ref{5.2-case} to $G[S^{\prime }_w]$ to obtain a set $U_w \subset S^{\prime }_w$ with $|U_w| \geq k_w+1$ and a probability distribution $\mathcal{D}_w$ on $[0.1,0.9]^{S^\prime_w}$ such that: 
    \begin{align}
        \label{eqn: internal-control-in-U_w}
        \baddd{D}{w}{S^\prime_w}{U_w}\leq |U_w| \cdot f(|U_w|),\ \ \ \text{where }f(x):=8\log_2x.
    \end{align}
\indent As in the proof of Theorem \ref{5.2-case}, if $k_w\leq 1$ then any set $U_w\subset S^\prime_w$ of size $2\geq k_i+1$ trivially satisfies \eqref{eqn: internal-control-in-U_w}, thus the above computations all make sense.

 We now set $U := \bigcup _{w \in W} U_w$ and $S^{\prime } := \bigcup _{w \in W} S^{\prime }_w$. Our distribution ${\cal D}$ will again be a product distribution, with ${\cal D}_w$ the forming factors. For each $w \in W$ we also take ${\cal E}_w $ to denote the uniformly 
constant distribution on $T_w$ given by ${\cal E}_w := {\cal U}_{T_w}$ and set $T := \cup _{w \in W} T_w$. We note that given distinct $w, w' \in W$ and $u \in U_w$, $u' \in U_{w'}$ we have $d_H^{T_w}(u) \geq 2t/3 \geq d_H^{T_w}(u') + t/3$. Therefore, by the choice of ${\cal E}_w$ and from Lemma 
\ref{lem: uniform-distribution-control} we find that: 
    \begin{align}
        \label{eqn: cross-control-from-T_w}
        \baddd {\cal E}{w}{T_w}{u,u'} \leq 9/t.
    \end{align}
\indent We also let ${\cal T}_R$ denote the trivial $R$-induced distribution, where $R := V(G) \setminus (S^{\prime } \cup T)$. Let ${\cal D}$ be the product distribution on $[0.1,0.9]^{S^{\prime }} \times [0.1,0.9]^{T} \times [0.1,0.9]^{R} = [0.1,0.9]^{V(G)}$ below:
    \begin{align*}
        {\cal D} 
                := 
        \big ( \prod _{w \in W} {\cal D}_w \big ) 
                \times
        \big ( \prod _{w \in W} {\cal E}_w \big )
                \times 
        {\cal T}_R.
    \end{align*}
    \indent To complete the proof, we are only left to lower bound $|U|$ and upper bound $\bad {U}$. For the lower bound, using \eqref{eqn: lower bound on union of S} and that $\log_2n \leq 2\sqrt{2}\log_2k$ from \eqref{eqn:log-corres}, we obtain:
    \begin{align*}
        |U| 
            =
        \sum  _{w \in W} \big | U_w \big |
            \geq 
        \sum _{w \in W} k_w 
            \geq 
        \sum _{w \in W} |S_w^{\prime }| \cdot \frac {k}{n} 
            = 
        \frac {k}{n} \bigg ( \big | \bigcup _{w \in W} S_w^{\prime } \big | \bigg ) 
            \geq 
        \frac {k}{2^{19} \log ^2k}=\frac{t}{9}.
    \end{align*}
\indent For the upper bound on $\bad {U}$, we have:
    \begin{align*}
        \bad {U} 
                &= 
        \sum _{w \in W} \bad {U_w} + \sum _{\{w,w'\} \subset W} \bad {U_w, U_{w'}}\\
                & \leq 
        \sum _{w \in W} \baddd {\cal D}{w}{S_w^{\prime }} {U_w} 
                + 
        \sum _{\{w,w'\} \subset W} \baddd {\cal E}{w}{T_w}{U_w, U_{w'}}\\
                & \leq 
        \sum _{w \in W} \baddd {\cal D}{w}{S_w^{\prime }} {U_w} 
                + 
        \sum _{\{w,w'\} \subset W} |U_w|\cdot |U_{w'}| \cdot \max _{(u,u') \in U_w \times U_{w'}} 
        \baddd {\cal E}{w}{T_w}{u_w, u_{w'}} \\
                & \leq 
        \sum _{w \in W} |U_w| \cdot f(|U_w|)
                + 
        \sum _{\{w,w'\} \subset W} |U_w|\cdot |U_{w'}| \cdot \bigg ( \frac {9}{t} \bigg ).
    \end{align*}
\indent The first inequality here follows Lemma \ref{lem: product-dist-lem}, the second is immediate from the definition of $\baddd {D}{}{S}{U,V}$, whereas the third one holds by \eqref{eqn: internal-control-in-U_w} and \eqref{eqn: cross-control-from-T_w}.

Choose a smallest subset $W^\prime := \{w_1,\ldots, w_M\}$ of $W$ so that $|\bigsqcup_{w\in W^\prime} U_w|\geq t/9$. If $W^\prime=\{w^\prime\}$ for some $w^\prime\in W$ then we are done by simply taking $U=U_{w^\prime}$ since $\bad{U_{w^\prime}}\leq |U_{w^\prime}|f(|U_{w^\prime}|)$. Otherwise we can assume that the sequence $U_i := U_{w_i}$ is non-increasing in size with $i$, i.e. that $|U_1| \geq |U_2| \geq ... \geq |U_M|$. Setting $U_{< i} := \bigcup _{j < i} U_i$, we immediately see from our choice of $W^\prime$ that $|U_{<i}|\leq t/9$. Our bound on $\bad {U}$ from above thus gives: 
    \begin{align*}
        \bad {U} 
            \leq 
        \sum _{i\in [M]} |U_i| \cdot f\big ( |U_i| \big ) 
            + 
        \sum _{i\in [M]}  \bigg ( \frac{9|U_{<i}|}{t} \bigg ) |U_i|
            \leq 
        \sum _{i\in [M]} |U_i| \cdot f\big ( |U_i| \big ) 
            + 
        \sum _{i\in [2,M]}|U_i|.
    \end{align*}
\indent For each $i\geq 2$ we have $|U_i| \leq |U_{< i}|$ from the ordering and so by applying Lemma \ref{log2ineq} we get $|U_{<i}| \cdot f(|U_{<i}|) + |U_i| \leq |U_{<i+1}| \cdot f(|U_{< i+1}|)$. Repeatedly applying this as $i$ increases gives us $\bad {U} \leq |U| \cdot f(|U|)$, noting that $U_{<m+1} = U$. This completes the proof.
\end{proof}

Let us remark by combining the two cases in the proof above that $|U|\geq 2^{-25}k\log_2^{-2}(k+1)$.\\
We are finally able to prove Theorem \ref{thm: NT-distribution-version}. The proof is very similar to that of Theorem \ref{5.2-case} but, as the details are involved, for completeness we will go through it with care. 

\begin{proof} [\textbf{Proof of Theorem
\ref{thm: NT-distribution-version}}] 
    We will prove a slightly 
    more convenient statement, namely that given the hypothesis there is a set $U \subset V(G)$ and a distribution ${\cal D}$ on $[0.1,0.9]^{V(G)}$ such that $|U|\geq 2^{-26}(k+k^{3/4})\log_2^{-2}(k+1)$ and $\bad{U}\leq 8|U|\log_2|U|$. 
    We will prove this by induction on $|V(G)|$. Note that the 
    theorem trivially holds in the first case where the hypothesis applies, when $n=20000$ and $k =1$ (taking $U$ to be 
    any sets of size $1$ and ${\cal D}$ the trivial distribution).
    Also, as in Theorem \ref{longest-ever}, when $k$ is small there is nothing to prove, so we can assume $k\geq 2^{25}$. 

Let $V_0$ be a largest vertex set of $G$ such that $|\diver{u,v}|\geq 2k^{3/2}$ for all $u,v\in V_0$. If $|V_0|\geq k+1$ then we are done by Lemma \ref{lem: div-approx}, otherwise assume that $V_0=\{v_1,v_2,\ldots, v_L\}$ for some $L\leq k$ and now for each $i\in [L]$ define the set $V_i:=\{v\in V(G): |\diver{v,v_i}| < 2k^{3/2}\}$. Due to the maximality of $S_0$ we get $V(G)=\bigcup_{i=1}^L V_i$. 
The proof splits into the two already familiar cases:\vspace{2mm}

\noindent {\textbf {Case I:}}  Every $j\in [L]$ with $d_G(v_j)\in [{nk^{-1/3}},n-1-nk^{-1/3}]$ 
satisfies $|V_j| \leq 3k$.\vspace{2mm}

We have seen that there are at most $3k^2$ vertices of $G$ that do not lie in a set $V_j$ of size at least $3k$. Moreover, $d_G(v_i)-2k^{3/2}< d_G(v)< d_G(v_i)+2k^{3/2}$ for all $i\in[L]$ and $v\in V_i$. Therefore, at least $n - 3k^2$ vertices $v\in V(G)$ have degree $d_G(v) \notin [nk^{-1/3}+2k^{3/2},n-1-nk^{-1/3}-2k^{3/2}]$. Therefore for all such vertices we have $d_G(v) \leq nk^{-1/3}+2k^{3/2}$ or $d_G(v) \geq n-1-nk^{-1/3}-2k^{3/2}$. We will assume that at least half of these vertices satisfy the first condition, as otherwise we can follow an identical argument working with $\overline{G}$ instead. Consequently we find a set $V\subset V(G)$ with $|V| \geq (n-3k^2)/2\geq 12n/25$ and $\Delta(G[V])\leq nk^{-1/3}+2k^{3/2}<4|V|k^{-1/3}$. Moreover, $|V|\geq 8000k^2$ and $\hom (G[V]) \leq \hom (G) \leq n/25k \leq |V|/12k$, hence we can apply Theorem \ref{longest-ever} to $G[V]$ to obtain a distribution ${\cal D}_1$ on $[0.1,0.9]^V$ and a vertex set $U \subset V$ with $\bad{U}\leq 8|U|\log_2|U|$ and $|U|\geq 2^{-25}k\log_2^{-2}(k+1)>2^{-26}(k+k^{3/4})\log_2^{-2}(k+1)$ . We also set ${\cal D}_0 := {\cal T}_{V(G) \setminus V}$, i.e. the trivial distribution induced by $V(G) \setminus V$, and let ${\cal D} := {\cal D}_0 \times {\cal D}_1$ denote the product distribution on $[0.1,0.9]^{V(G)}$. By making use of Lemma \ref{lem: product-dist-lem} we can, once again, obtain $\bad {U} \leq \baddd {D}{1}{V}{U} \leq 8|U|\log_2|U|$, as required.\vspace{2mm}

\noindent {\textbf {Case II:}}  There is $j\in [L]$ such that 
$d_G(v_j)\in [{nk^{-1/3}},n-1-nk^{-1/3}]$ and $|V_j| \geq 3k$.\vspace{2mm}

As in Theorem \ref{5.2-case}, pick a subset $V$ of $V_j$ of size $3k$ such that $v_j\in V$, then set  ${X_1}:=N(v_j)\setminus V$ and ${X_2}=:V(G)\setminus(V\cup N(v_j))$. We note that both $|{X_1}|, |{X_2}| \geq nk^{-1/3}-3k$. The same double counting argument from Theorem \ref{5.2-case} works here to give us the sets $Y_1=\{u\in X_1:d_G^V(u)\geq 2k\}$ and $Y_2=\{u\in {X_2}:d_G^V(u)\leq k\}$, both of size at least $nk^{-1/3}-3k-6k^{3/2}>4096k^{3/2}>2^{48}$, such that $|X_i\setminus Y_i|\leq 6k^{3/2}$ for each $i\in\{1,2\}$. Since $V(G) = V \cup X_1 \cup {X_2}$ is a partition, this shows that $Z:=V(G)\setminus(Y_1 \cup Y_2)$ satisfies $|Z|\leq 2\cdot 6k^{3/2}+3k\leq 15k^{3/2}$. 

\indent To complete the proof we will apply the induction hypothesis to both $Y_1$ and $Y_2$. Let $t_i:=|Y_i|$ for
 $i\in \{1,2\}$ and set $k_i:=k(t_i/n) \leq k$. This gives us $\hom (G[Y_i]) \leq \hom (G) \leq n/25k=t_i/25k_i$. We also have $t_ik_i^{-2}\geq k_i^{-1}(t_i/k_i)=k_i^{-1}(n/k)\geq nk^{-2}\geq 20000$. Therefore, for $i=1,2$ we can apply the induction hypothesis to $G[Y_i]$ to find a probability distribution $\mathcal{D}_i$ on $[0.1,0.9]^{Y_i}$ and a set $U_i\subset Y_i$ which satisfies $|U_i|\geq 2^{-26}(k_i+k_i^{3/4})\log_2^{-2}(k_i+1)$ and:
   \begin{align}
        \label{eqn: final inductive-control}
    \baddd{D}{i}{}{U_i}\leq |U_i| \cdot f(|U_i|),\ \ \ \text{where }f(x):=8\log_2x.
   \end{align}
   \indent Let us remark that if $k_i< 2^{33}$ then any set $U_i\subset Y_i$ of size $2\geq 2^{-25}k_i\log_2^{-2}(k_i+1)$ trivially satisfies \eqref{eqn: final inductive-control}, as already noted many times before, thus the above computations all make sense.\\
   \indent We can also assume that $\max \{|U_1|, |U_2| \} \leq k$, as otherwise the theorem follows immediately by just taking $U$ to equal one of these sets.
   
\indent We will also let ${\cal D}_0 := {\cal U}_V$ denote the uniformly constant distribution on $[0.1,0.9]^V$ and let ${\cal D}_3 := {\cal T}_{Z}$ denote the trivial $Z$-induced distribution. We now set $U:=U_1\cup U_2$ and let ${\cal D}$ denote the product distribution $\prod _{i\in [0,3]} {\cal D}_i$ on $[0.1,0.9]^{V} \times \prod _{i\in [2]} [0.1,0.9]^{Y_i}\times [0.1,0.9]^Z= [0.1,0.9]^{V(G)}$. 

Note that $d^V_G(u) \geq 2k \geq d_G^V(v) + k$ for all $u \in Y_1$ and $v \in Y_2$, by definition of $Y_1$ and $Y_2$. It then follows from Lemma \ref{lem: uniform-distribution-control} that for all such vertices we have:
    \begin{align}
        \label{eqn: final split-control-to-V}
        \baddd {D}{0}{V}{u,v} \leq \frac {3}{k}.
    \end{align}
    \indent As $V(G)=Y_1\cup Y_2\cup Z$ is a partition and $|Z|\leq 15k^{3/2}$, we get $t_1+t_2\geq n-15k^{3/2}$, therefore $k_1+k_2\geq k-15n^{-1}k^{5/2}\geq k-\sqrt{k}/300$. Moreover, recall $t_i\geq nk^{-1/3}-3k-6k^{3/2}>nk^{-1/3}/2$, hence $k_i = k(t_i/n) \geq k^{2/3}/2$ for $i\in\{1,2\}$. By using Lemma \ref{oprimization-map} we obtain that:
    $$k_1^{3/4}+k_2^{3/4}\geq \frac{\sqrt{k}}{\sqrt[4]{8}}+\left(k-\frac{\sqrt{k}}{300}-\frac{\sqrt[3]{k^2}}{2}\right)^{3/4}>
    \frac{\sqrt{k}}{2}+k^{3/4}\cdot \left(1-\frac{2k^{-1/3}}{3}\right)^{3/4}.$$
    \indent Using the inequalities $1-t\geq \exp(-2t)$ and $\exp(-t)\geq 1-t$, which hold for any $t\in [0,0.5]$ and in particular for $t=\Theta(k^{-1/3})$, we can further deduce that:
    \begin{align}\label{concavity-ineq}
     k_1^{3/4}+k_2^{3/4}\geq \frac{\sqrt{k}}{2}+k^{3/4}\cdot \exp\left(-k^{-1/3}\right)>\frac{\sqrt{k}}{2}+k^{3/4}-k^{-5/12}>k^{3/4}+\frac{\sqrt{k}}{300}.
     \end{align}
    \indent We are now in a position to lower bound the size of $U$:
\begin{align*}
    |U| = |U_1|+|U_2|
    &\geq \frac{1}{2^{26}}\cdot \left(\frac{k_1+k_1^{3/4}}{\log_2^2(k_1+1)}+\frac{k_2+k_2^{3/4}}{\log_2^2(k_2+1)}\right)\geq\\
     &\geq \frac{1}{2^{26}\log_2^2(k+1)} \cdot \big(k_1+k_1^{3/4}+k_2+k_2^{3/4}\big)\overset{\eqref{concavity-ineq}}{\geq} \\
    &\geq \frac{1}{2^{26}\log_2^2(k+1)} \cdot \left(k_1+k_2+\frac{\sqrt{k}}{300}+k^{3/4}\right)\geq \frac{k+k^{3/4}}{2^{26}\log_2^2(k+1)}.
\end{align*}
\indent Finally, we are able to estimate $\bad{U}$ as follows: 
\begin{align*}
    \bad{U}
       & = \sum_{\{u,v\}\subset U_1}\bad{u,v}+\sum_{\{u,v\}\subset U_2}\bad{u,v}+\sum_{(u,v)\in U_1 \times U_2}\bad{u,v}\\
    &=\bad{U_1}+\bad{U_2}+|U_1||U_2|\cdot \max _{(u,v)\in U_1 \times U_2} \big \{\bad{u,v} \big \}\\
    &\leq \baddd{D}{1}{}{U_1}+\baddd{D}{2}{}{U_2}+|U_1||U_2| \cdot \max _{(u,v)\in U_1 \times U_2} \big \{\baddd {D}{0}{V}{u,v} \big \}\\
    &\leq 
    |U_1|\cdot f(|U_1|) + |U_2| \cdot f(|U_2|) + \frac {3}{k} \cdot |U_1||U_2| \leq |U|\cdot f(|U|).
\end{align*}
\indent The final three inequalities here respectively follow from Lemma \ref{lem: product-dist-lem}, then from 
\eqref{eqn: final inductive-control} and \eqref{eqn: final split-control-to-V}, and lastly using that $\max \{|U_1|,|U_2|\} \leq k$ and Lemma \ref{log2ineq}. This completes the proof.
\end{proof}

\section{Distinct degrees in random graphs}
\label{sec: random graph distinct degrees}

In this section we will study $f(G(n,p))$, the number of distinct degrees which can be found in an induced subgraph of the Erd\H{o}s$-$R\'enyi random graph $G(n,p)$. Our results extend the estimates for the case of constant $p$ due to Bukh and Sudakov \cite{bukh} and to Conlon, Morris, Samotij and Saxton \cite{unpublished}. We restate Theorem \ref{thm: DD for G(n,p)} for the reader's convenience. 

\randomdistinctdegrees*

Although the estimation of $f(G(n,p))$ is quite natural in itself, we believe, as discussed in the concluding remarks, that the behaviour for $p \in [n^{-1/2}, 1/2]$ essentially determines the extremal relationship between $\hom (G)$ and $f(G)$ beyond the range of the Narayanan--Tomon conjecture, when $\hom (G) < n^{1/2}$. As a result, our calculations will focus on the case (i) of Theorem \ref{thm: DD for G(n,p)}. The next subsection contains the proof the upper bound on $f(G(n,p))$ in this case, whereas the second subsection contains the more difficult lower bound. In the final subsection we briefly indicate how to approach the case when $p \leq n^{-1/2}$.

\subsection{Upper bound on \texorpdfstring{$f(G(n,p))$}{f(G(n,p))}}

In this subsection we prove the upper bound on $f(G(n,p))$. Our approach closely follows that of Bukh and Sudakov (see Proposition 2.4 in \cite{bukh}), but we include the complete details, as the estimates are more involved in the sparse case.

\begin{prop}
Given $n \in {\mathbb N}$ and $p \in [n^{-1/2}, 1/2]$, one has $f\big(G(n,p)\big)=O\big(\sqrt[3]{pn^2}\big)$ whp.
\end{prop}

\begin{proof}
Suppose $G\sim G(n,p)$ has a subset $A\subset V(G)$ of size $a$ such that $G[A]$ has $8b$ distinct degrees, where $b = 16\sqrt[3]{pn^2}$. As at most $6b-1$ of our distinct degrees can lie in the interval $(pa-3b,pa+3b)$, either there are at least $b$ vertices of $A$ that have degree at least $pa+3b$ or at least $b$ vertices that have degree at most $pa-3b$. 

\indent We will assume first that we are in the former case, as this is the more intricate one. Let $B\subset A$ be a set of $b$ vertices which all have degree at least $pa+3b$. Let us now look at the number $e(A,B)$ of edges with one endpoint in $A$ and one in $B$ (those with both their endpoints in $B$ will be counted twice since $B\subset A$). We then have $pab+3b^2\leq e(A,B)=2e(B)+e(A\setminus B,B)$. As $|B| = b$ implies $e(B) < b^2$, we find that: \begin{align}
\label{eqn: unlikely inequality}
    e(A\setminus B, B) \geq pab + b^2 \geq p(a-b)b + b^2.
\end{align}
\indent Letting $F$ denote the event that there are sets $A$ and $B$ which satisfy \eqref{eqn: unlikely inequality}, it suffices to show that ${\mathbb P}(F) = o(1)$. To see this, first suppose that $16p(a-b)\geq b$. As $\bE[e(A\setminus B,B)]=p(a-b)b$, by using Chernoff's Inequality with $\delta={b}/{16p(a-b)}\leq 1$ we get that: 
\begin{align*}
\bP\left(e(A\setminus B,B)\geq pb(a-b)+b^2\right) & \leq 
\bP\left(e(A\setminus B,B)\geq pb(a-b)+2^{-4}b^2\right)\\
&\leq \exp\left(\dfrac{-b^3}{2^{10}p(a-b)}\right) 
\leq \exp\left(\dfrac{-b^3}{2^{10}pn}\right),
\end{align*}
where the final inequality uses that $a-b\leq a\leq n$. Therefore, the union bound implies that event $F$ can happen with probability at most: $$
{\mathbb P}(F) \leq 2^n\cdot \dbinom{n}{b}\cdot \exp\left(\dfrac{-b^3}{4\cdot pn}\right) 
    \leq 
2^{2n} \cdot \exp\left(\dfrac{-b^3}{2^{10}\cdot pn}\right).$$
\indent This tends to zero as $n \to \infty $, as $b \geq 16\sqrt [3]{pn^2}$.

\indent Now suppose instead that  $16p(a-b)< b$. As $e(A\setminus B,B)$ has binomial distribution, we have: $$\bP\left(e(A\setminus B,B)\geq pb(a-b)+b^2\right)\leq \dbinom{b(a-b)}{b^2}\cdot p^{b^2}\leq \left(\dfrac{2ep(a-b)}{b}\right)^{b^2}\leq 2^{-b^2}.$$
\indent Recalling that $p \geq n^{-1/2}$ and that $b \geq 16 \sqrt [3] {pn^2} \geq 
16\sqrt [3]{n^{-1/2}n^2} \geq 16\sqrt n$, using the union bound we find that the event $F$ occurs with probability at most:
$$ {\mathbb P}(F) \leq 2^n\binom {n}{b} 2^{-b^2} \leq 2^{2n}\cdot 2^{-(16\sqrt n)^2}.$$ 
\indent Hence, it follows again that ${\mathbb P}(F) = o(1)$ in this second case.

\indent Finally, if there is a set $B\subset A$ of $b$ vertices that all have degree at most $pa-3b$ in $G[A]$, then $e(A\setminus B,B)\leq e(A,B)\leq b(pa-3b) \leq pb(a-b)- b^2$. If $p(a-b)< b$ then it is clear that such a set $B$ exists with $0$ probability since we cannot have a negative number of edges. Otherwise we can simply apply Chernoff's Inequality for $\delta =\dfrac{b}{p(a-b)} \leq 1$ to get: $$\bP\bigg(e(A\setminus B ,B )\leq pb(a-b)-b^2\bigg)\leq \exp\left(-\dfrac{b^3}{2p(a-b)}\right)\leq \exp\left(-\dfrac{b^3}{2pn}\right),$$ 
where the last inequality follows as $a-b\leq a\leq n$. We have seen before that the union bound gives us probability of at most $2^n\cdot \binom{n}{b}\cdot \exp\left(-b^3/2pn\right)$ for such a set $B$ to exist and we have shown in the previous case that this probability tends to $0$ as $n\to \infty$.
\end{proof}

\subsection{Lower bound on \texorpdfstring{$f(G(n,p))$}{f(G(n,p))}}

We now focus on proving our sharp lower bound for $f(G(n,p))$. Before we start, we will present a few results that will help us along the way.\\
\indent Given $D>0$ and a graph $G$, we call a set $U\subset V(G)$ $D$-\emph{diverse} if it is $D$-diverse to $V(G)$. We say that the graph $G$ is $D$-\emph{diverse} if $V(G)$ is $D$-diverse (see Section 4 before Lemma \ref{lem: blended-distribution-control}).

\begin{prop}\label{deg-ord-np}
If $p\gg \log n/n$ then all vertices of $G(n,p)$ have degree asymptotic to $np$ whp. In particular, whp they all have degrees less than $2np$.
\end{prop}

\begin{proof}
Let $u$ be a vertex of $G(n,p)$. Then $d_{G(n,p)}(u)\sim Bin(n-1,p)$, so we can apply Chernoff's Inequality for $\delta =3\sqrt{\dfrac{\log n}{np}}$ to get $\bP\big(|d(u)-np|\geq 3\sqrt{np\log n}\big)\leq 2n^{-9/2}$. The result now follows by the union bound. The last part is a consequence of the fact that $\delta \leq 1$. 
\end{proof}

\begin{lem}\label{random-diversity}
If $p\gg \log n/n$ and $p\leq 1/2$ then whp $G(n,p)$ is $p(n-1)$-diverse.
\end{lem}

\begin{proof}
Let us first notice that $|\text{div}(u,v)|\sim Bin(n-2,q)$ for all distinct $u,v\in V(G)$, where $q:=2p(1-p)$. This holds as every vertex of $\text{div}(u,v)$ has to be in $N(u)\setminus N(v)$ or in $N(v)\setminus N(u)$ and these two possibilities represent disjoint events that happen with probability $p(1-p)$.\\
\indent Note that $q\geq p\gg \log n/n$ and so we can apply Chernoff's Inequality for $\delta=3\sqrt{\dfrac{\log n}{nq}}$ to obtain $\bP\big(|\text{div}(u,v)|\leq nq-3\sqrt{nq\log n}\big)\leq n^{-9/2}$. By the union bound, we therefore deduce that $|\text{div}(u,v)|>nq(1-\delta)$ for any $u,v\in V(G)$ whp. Since $q\geq p$ and $\delta\to 0$, we get that $G(n,p)$ is $p(n-1)$-diverse whp in this case.   
\end{proof}

\begin{lem}\label{bded-u-degree}
Let $n \in {\mathbb N}$ and $p \in [n^{-1/2}, 1/2]$ and suppose $G\sim G(n,p)$ and let $V(G):=U\sqcup S$ be a vertex partition such that $\sqrt {n} \leq 4 |U| \leq pn$. Then, with high probability, there is a subset $W\subset S$ such that $U$ is $pn/3$-diverse to $W$ and $d^U_G(w)\leq 10p|U|$ for all $w\in W$.   
\end{lem}

\begin{proof}
Define the set $S_B:=\{v\in V:d_G^U(v)\geq 10 p|U|\}$. For any $v\in S$ the random variable $d_G^U(v)$ has distribution $Bin(|U|,p)$, hence, as $p|U|\geq 2.5$, we deduce that: $$\Tilde{p}:=\bP(v\in S_B)\leq (e/10)^{10p|U|}<3^{-2.5}<1/10. $$
\indent Now for each subset $W\subset S$ we see that $|W\cap S_B|$ has distribution $Bin(|W|,\Tilde{p})$. Therefore, for any $u_1,u_2\in U$ we can deduce by using Theorem \ref{binomial-bound} and Lemma \ref{random-diversity} that:
$$\bP\left(|\text{div}(u_1,u_2)\cap S_B|>\frac{|\text{div}(u_1,u_2)|}{3}\right)\leq \left(3e\Tilde{p}\right)^{|\text{div}(u_1,u_2)|/3}\leq \left(\frac{9}{10}\right)^{\sqrt{n}/4}.$$
\indent Call a pair $\{u_1,u_2\}\subset U$ of vertices \emph{big} if $|\text{div}(u_1,u_2)\cap S_B|>|\text{div}(u_1,u_2)|/3$. By using the union bound we immediately deduce that $U$ contains a \emph{bad} pair of vertices with probability at most $p^2n^2/8\cdot (9/10)^{\sqrt{n}/4}\to 0$ as $n\to \infty$. Set now $W:=S\setminus S_B$ and note that whp for all distinct $u_1,u_2\in U$ we have $|\text{div}(u_1,u_2)\setminus S_B|>2|\text{div}(u_1,u_2)|/3\geq 2p(n-1)/3$. We then get:
$$\big|N_G^W(u_1)\triangle N_G^W(u_2)\big|\geq \big|\text{div}(u_1,u_2)\setminus S_B\big|-|U|\geq \frac{2p(n-1)}{3}-\frac{\sqrt{n}}{4}>\frac{pn}{3}.$$
\indent The second property follows directly from the definition of $W$ and so our result is proved.
\end{proof}

  \begin{dfn}
Let $G$ be a $n$-vertex graph and let $0<p\leq 1/2$. We call a set $U$ of vertices of $G$ $p$\emph{-convenient} if $d_G(u)\leq 2pn$ for all $u\in U$ and there is a set $W\subset V(G)\setminus U$ such that $U$ is $pn/3$-diverse to $W$ and $d^U_G(w)\leq 10p|U|$ for all $w\in W$.
  \end{dfn}
  
  \indent We now expose the randomness in $G(n,p)$ and obtain a fixed graph $G$. According to Proposition \ref{deg-ord-np} and Lemma \ref{bded-u-degree}, we may assume that $G$ contains a $p$-convenient set $U$ of size $\sqrt[3]{pn^2}/4$.
  
  \indent At this point, the reader might have already noticed that the $p$-convenient conditions fit in very well with 
  those from Lemma \ref{lem: blended-distribution-control}. Indeed, the set $U$ is $pn/3$-diverse to 
  $W$ and $10p$-balanced, so the hypothesis of the lemma is satisfied. The most natural thing to do would now be to apply the lemma with the blended distribution $\mathcal{B}_\beta(U,W)$. However, 
  in order to obtain a set of $\Theta (|U|)$ degrees, we would like the first term in the RHS of \eqref{eqn:bad-eqn-control} to be of order $|U|^{-1}$, which forces $\beta:=\Theta\left(|U|/pn\right)$. This would then make the second term in the RHS of \eqref{eqn:bad-eqn-control} to be a constant, so it seems that we cannot get the desired `bad' control. One can obtain weaker bounds on $f(G(n,p))$ by altering 
  the parameters here, but there is an unavoidable loss as things stand.
  
  \indent There is though a way around this issue. In Lemma \ref{lem: blended-distribution-control} we solve the problem of coordinates lying outside $[0.1,0.9]$ by dealing with each pair of vertices $\{u_1,u_2\}\subset U$ individually. However, in certain situations it is possible to show that many vertices 
  $u_1 \in U$ are \emph{simultaneously} good for all pairs $\{u_1,u_2\}\subset U$. 
   The crucial twist here is that the diversity term $D:=pn/3$ satisfies $D=\Omega(\Delta(G))$. This allows us to guarantee that a fixed vertex $u_1 \in U$ whp is likely to have 
  no neighbours in $\text{div}(u_1,u_2)$ whose coordinates are 'outliers', and this happens \emph{for all} $u_2 \in U$. The approach here 
  builds upon that of Jenssen, Keevash, Long and Yepremyan \cite{JKLY}.
  
  A slight change in our notation will be convenient below. Given a graph $G$ with vertex partition $V(G)=U\sqcup W$ and a probability vector $\laur{p} = (p_w)_{w \in W} \in [0, 1]^{W}$, we write $G(\laur{p})$ to denote the probability space on the set of induced subgraphs of $G$ \emph{that contain} $U$, where for each vertex set $S \subset W$, the induced subgraph $G[U\cup S]$ is selected with probability $\prod _{v \in S}p_v \prod _{v \in W \setminus S} (1-p_v)$.

\begin{prop}\label{prop-expected-deg}
Let $n \in {\mathbb N}$, $p \in [n^{-1/2}, 1/2]$ and let $G$ be an $n$-vertex graph with a $p$-convenient set $U\subset V(G)$ of size $\sqrt[3]{pn^2}/4$. Then there is a vector $\textbf{\ul{p}}\in [0.1,0.9]^{V(G)\setminus U}$ and a set $U^\prime\subset U$ with $|U^\prime|\geq |U|/500$ so that $\big|\bE[d_{G(\laur{p})}(u_1)]-\bE[d_{G(\laur{p})}(u_2)]\big|\geq 1$ for all distinct $u_1,u_2\in U^\prime$. 
\end{prop}

\begin{proof}
We may assume that $n$ is large enough so that all asymptotic bounds hold. First set $\beta:= |U|/5pn<0.1$ and let $S\subset V(G)\setminus U$ be such that $U$ is $pn/3$-diverse to $S$ and $d^U_G(v)\leq 10p|U|$ for all $v\in S$. As in the proof of Lemma \ref{lem: blended-distribution-control}, for each $u\in U$ define the random vector $\laur{q}^u$ on $\mathbb{R}^S$ by $\laur{q}^{u}:=\laur{p}^\prime-\alpha_u\cdot \proj {S}{u}{}$, where $\laur{p}'$ is the vector from before the truncation in the definition of the blended distribution $\mathcal{B}_\beta(U,S)$. Recall we do this so that $\laur{q}^u$ is independent of $\alpha_{u}$. 

We call a vertex $u\in U$ \emph{good} if there are at most $d_G^S(u)/25$ coordinates $v\in S\cap N(u)$ so that $\laur{q}_v^u\notin [0.2,0.8]$. Let $U^g\subset U$ denote the set of \emph{good} vertices.\\
\indent We claim that $\bP(|U^g|\geq |U|/2)>1/2$. To prove this, take $u\in U$ and note that $\laur{q}^u_v$ is a sum of at most $10p|U|$ uniform independent random variables, thus by Hoeffding's inequality:
$$\bP\big(\laur{q}_v^u\notin [0.2,0.8]\big)=\bP\big(|\laur{q}_v^u-1/2|>0.3 \big)\leq 2\exp\left(\dfrac{-50\cdot 0.09\cdot p^2n^2}{10p|U|^3}\right)=2e^{-7.2}<\frac{1}{100}.$$
\indent We deduce that the expected number of coordinates $v\in V\cap N(u)$ with $\laur{q}_v^u\notin [0.2,0.8]$ is at most $d_G^V(u)/100$. By Markov we get that the vertex $u$ is not good with probability less than $1/4$. Therefore, the expected number of vertices $u\in U$ that are not good is at most $|U|/4$ and the claim follows from a simple application of Markov's Inequality.

We now set $T:=V(G)\setminus(S\cup U)$ and let $\mathcal{T}_T$ denote the trivial $T$-induced distribution. Take $\mathcal{D}$ to be the product distribution $B_\beta(U,S)\times \mathcal{T}_T$ on $[0.1,0.9]^{S\cup T}$. Given distinct vertices $u_1,u_2\in U$, we let $E_{u_1,u_2}$ denote the event that $\left|\bE_{\laur{p}\sim \mathcal{D}}\big[d_{G(\laur{p})}(u_1)-d_{G(\laur{p})}(u_2)\big]\right|\leq 1$. Moreover, since $U$ is $pn/3$-diverse to $S$, we know that either $|N_G^S(u_1)\setminus N_G^S(u_2)|\geq pn/6$ or $|N_G^S(u_2)\setminus N_G^S(u_1)|\geq pn/6$. We set $m_S(u_1,u_2):=u_1$ in the first case and $m_S(u_1,u_2):=u_2$ in the second one.

\indent Our next claim is that $\bP\big(E_{u,u^\prime}\ |\ m_S(u,u^\prime)\in U^g\big)\leq 120|U|^{-1}$ for all $u\neq u^\prime$ in $U$. To prove it, we can assume that $u=m_S(u,u^\prime)$. As $u\in U^g$, at most $2pn/25$ vertices in $N_G(u)$ represent coordinates $v$ such that $\laur{q}^u_v\notin [0.2,0.8]$. Therefore, we can find a subset $Y\subset N_G^S(u)\setminus N_G^S(u^\prime)$ of size $pn/12$ such that $\laur{q}^u_v\in [0.2,0.8]$ for all $v\in Y$. Since $\laur{p}^\prime_v=\laur{q}^u_v+\alpha_u \textbf{u}_v$ and $|\alpha_u|<0.1$, we deduce that no $Y$-coordinate of $\laur{p}^\prime$ gets truncated when creating $\laur{p}\sim B_\beta(U,S)$. Condition now on any choice of $\bm{\alpha}:=(\alpha_{w})_{w\neq u}$ such that $u\in U^g$ and note that $\alpha_u$ is independent of it.\\
\indent By looking at the following expression (when $\laur{p}\sim \mathcal{D}$) as a function of $\alpha_u$:
\begin{align*}
{\mathbb E}[ d_{G(\laur{p})}(u)] - 
    {\mathbb E}[ d_{G(\laur{p})}(u^\prime)] &= \text{ constant }+
 {\mathbb E}[ d_{G(\laur{p})}^S(u)] - 
    {\mathbb E}[ d_{G(\laur{p})}^S(u^\prime)]\\
    &= \text{ constant }+ (\proj {S}{u}{} - \proj {S}{v}{}) \cdot \proj{S}{\laur{p}}{}    
\end{align*}
we observe that $E_{u,u^\prime}$ holds provided that, conditioned on $\bm{\alpha}$, this difference lies in an interval of length $2$. The same argument as in Lemma \ref{lem: blended-distribution-control} gives us that $E_{u,u^\prime}|\bm{\alpha}$ happens with probability at most $24(pn\beta)^{-1}=120|U|^{-1}$. The claim follows from the law of total probability.

\indent To complete the proof, we consider the graph $J$ on the vertex set $U^g$ where $u_1u_2\in E(J)$ if $E_{u_1,u_2}$ holds. By the second claim we get $\bE[{e(J)}]\leq 120|U|^{-1}\cdot |U^g|(|U^g|-1)/2< 60|U|$, thus by Markov $\bP\big(e(J)>120|U|\big)<1/2$. It follows that $\bP\big(e(J)\leq 120|U|\big)>1/2$ and recall that $\bP(|U^g|\geq |U|/2)>1/2$. Therefore, with positive probability, we can choose $\laur{p}\sim \mathcal{D}$ such that $|U^g|\geq |U|/2$ and $e(J)\leq 120|U|$. For such a choice, the average degree of the resulting graph $J$ is $4e(J)/2|U_g|\leq 4e(J)/|U|\leq 480$. Thus, by Tur\'{a}n's Theorem $J$ has an independent set of size at least $|U|/500$. This independent set in $J$ is precisely what we required. 
\end{proof}\vspace{0.5mm}

\begin{prop}\label{random-to-discrete-degree}
 Let $G$ be a $n$-vertex graph and let $p \in [n^{-1/2}, 1/2]$. Suppose that there is a $p$-convenient set $U\subset V(G)$ in $G$, a vector $\textbf{\ul{p}}\in [0.1,0.9]^{V(G)\setminus U}$ and a vertex subset $U^\prime\subset U$ so that $\big|\bE[d_{G(\laur{p})}(u_1)]-\bE[d_{G(\laur{p})}(u_2)]\big|\geq 1$ for all distinct $u_1,u_2\in U^\prime$. Then $f(G)=\Omega(|U^\prime|)$.
\end{prop}

\begin{proof}
We can assume $n$ is sufficiently large. Let $H$ be a random induced subgraph selected according to $G(\laur{p})$ and define for it the following sets:
\begin{align*}
&B=\{u\in U^\prime: \big|d_H(u)-\bE[d_{G(\laur{p})}]\big|\leq \sqrt{2pn}\},\\ 
&P=\{\{u,u^\prime\}\subset U^\prime: \big|\bE[d_{G(\laur{p})}(u)]-\bE[d_{G(\laur{p})}(u^\prime)]\big|\leq 2\sqrt{2pn}\},\\
&J=\{\{u,u^\prime\}\in P:d_H(u)=d_H(u^\prime)\}.
\end{align*}
\indent Our first claim is that $\bP(|B|\geq|U^\prime|/2)\geq 1/2$. To prove it, we start by estimating $|B|$. For any $u\in U^\prime$ we have $\mathbb{V}\text{ar}\big(d_{G(\laur{p})}(u)\big)=\sum_{v\sim u} p_v(1-p_v)\leq pn/2$, thus Chebyshev's Inequality implies that $\bP(u\notin B)\leq 1/4$. It follows that $\bE[|U^\prime\setminus B|]\leq |U^\prime|/4$, so by Markov's inequality we get that $\bP(|U^\prime\setminus B|\geq |U^\prime|/2)\leq 1/2$, which is equivalent to our claim.

\indent We now want to estimate $|J|$. First note that the separation in expected degree for $U^\prime$ implies that $|P|\leq 2|U^\prime|\sqrt{2pn}$. Each $\{u,u^\prime\}$ belongs to $J$ with probability $\bP\big(d_H(u)-d_H(u^\prime)=0\big)$, which we claim is $O\left(1/\sqrt{pn}\right)$. This happens because $d_H(u)-d_H(u^\prime)=\sum \xi_v X_v$, where the sum is taken over all $v\in \text{div}(u,u^\prime)\setminus U$, $\xi_v\in\{-1,1\}$ and $X_v\sim Be(p_v)$ measures whether $v\in V$ is picked as a vertex of $H$ or not. As $U$ is $pn/3$-diverse to some subset $S\subset V(G)\setminus U$, we deduce that $|\text{div}(u,u^\prime)\setminus U|\geq |N_G^S(u)\triangle N_G^S(u^\prime)|\geq pn/3$, so we can apply Theorem \ref{problitoff} to prove the previous claim. Therefore $\bE[|J|]\leq |P|\cdot \displaystyle\max_{\{u,u^\prime\}\in P}\bP\big(d_H(u)=d_H(u^\prime)\big)=O(|U^\prime|)$.\\
\indent It follows that $\bP\big(|J|=O(|U^\prime|)\big)>1/2$ by Markov, so together with the first claim, we are able to deduce that both $|J|=O(|U^\prime|)$ and $|B|\geq|U^\prime|/2$ happen with positive probability. The end of the proof follows the same idea as before: make a choice of $H$ for which this happens and by Tur\'{a}n's Theorem the graph $J[B]$ obtained by building edges between the vertices of $B$ which have equal degree in $H$ has an independent set of size $\Omega(|U^\prime|)$. This set must consist of vertices with distinct degrees in $H$, as if $u,u^\prime\in B$ and $d_H(u)=d_H(u^\prime)$ then $\{u,u^\prime\}\in P$ and so $\{u,u^\prime\}\in J$, which represents an edge in $J[B]$.
\end{proof}

With all these ingredients, we are finally able to prove the following:

\begin{thm}
Given $n \in {\mathbb N}$ and $p \in [n^{-1/2}, 1/2]$, one has $f\big(G(n,p)\big)=\Omega\big(\sqrt[3]{pn^2}\big)$ whp.
\end{thm}

\begin{proof}
We expose the randomness in $G(n,p)$ and thus move to a fixed graph $G$. According to Proposition \ref{deg-ord-np} and Lemma \ref{bded-u-degree}, we can find a $p$-convenient set $U$ in $G$ of size $\sqrt[3]{pn^2}/4$. We then apply Proposition \ref{prop-expected-deg} to find a vector $\textbf{\ul{p}}\in [0.1,0.9]^{V(G)\setminus U}$ and a subset $U^\prime\subset U$ of size $\Omega\left(\sqrt[3]{pn^2}\right)$ so that $\big|\bE[d_{G(\laur{p})}(u_1)]-\bE[d_{G(\laur{p})}(u_2)]\big|\geq 1$ for all distinct $u_1,u_2\in U^\prime$. Lastly, Proposition \ref{random-to-discrete-degree} allows us to convert a constant proportion of the distinct expected degrees in $U^\prime$ to genuine distinct degrees, thus completing the proof.
\end{proof}

 \subsection{ \texorpdfstring{$f(G(n,p))$ when $p \ll n^{-1/2}$}{f(G(n,p)) when p is small}}
 
 For completeness, in this subsection we discuss the behaviour of $f(G(n,p))$ for $p \ll n^{1/2}$. 
 First, to see that there is a change in behaviour here over the range $p \in [n^{-1/2},1/2]$, 
 note that if $G\sim G(n,p)$ with $\log n/n\ll p\ll n^{-1/2}$ then a simple concentration argument combined with the union bound shows that whp $d_G(u)=O(pn)$ for every vertex $u\in V(G)$, showing that  $f(G)=O(pn) = o(\sqrt [3] {pn^2})$ in this case. 
 
 As indicated in Theorem \ref{thm: DD for G(n,p)} (ii), in this regime the maximum degree of $G(n,p)$ is a key parameter. The following simple proposition is useful here.
 
 \begin{prop}
    \label{prop: sparse G(n,p) bound}
    Let $G$ be an $n$-vertex graph and let $U\subset V(G)$ with $|U| = k$  such that 
    $|N_G(u) \setminus \big ( U \cup \cup _{u'\in U\setminus \{u\}} N(u') \big )| \geq k$ for all $u \in U$. 
    Then $f(G) \geq k$.
 \end{prop}
 
 \begin{proof}
    Let $U := \{u_1,u_2\ldots, u_k\}$ such that $|N_G(u_i) \cap U|$ is non-decreasing with $i$. For each 
    $i\in [k]$ take $S_i \subset N_G(u_i) \setminus \big ( U \cup \cup _{u' \in U\setminus \{u\}} N(u') \big )$ 
    with $|S_i| = i\ -$ by the hypothesis such sets exist. It is now easy to see that the degrees of 
    the vertices $u_1,u_2,\ldots, u_k$ are strictly increasing in the induced 
    subgraph $G \big [U \cup (\cup _{i\in [k]} S_i) \big ]$, giving $f(G) \geq k$, as required.
 \end{proof}
 
 The following observations show that $f(G(n,p)) = \Theta \big ( \Delta ( G(n,p) ) \big )$ in this regime.
    \begin{itemize}
        \item [(i)] If $\log n / n \ll p \leq n^{-1/2}$ then whp $d_G(u) \in [pn/2, 2pn]$ for all $u\in V(G)$, therefore we deduce that $\Delta (G(n,p)) = \Theta (np)$;
        \item [(ii)] If $\log n / n \ll p \leq n^{1/2}$ then given any fixed set $U \subset V(G(n,p))$ with 
        $|U| = pn/8$ and $u \in U$ we have ${\mathbb E} \big [N_G(u) \setminus (U \cup \cup _{u' \in U \setminus \{u\}} N(u') ) \big ] = p(n-|U|)(1-p)^{|U|} \geq pn/4$. Chernoff's Inequality then 
        applies to give us that
        $|N_G(u) \setminus (U \cup \cup _{u' \in U \setminus \{u\}} N(u') )| \geq pn/8 = |U|$ for all $u\in U$ whp. Proposition \ref{prop: sparse G(n,p) bound} together with (i) then gives
        $f(G) \geq |U| = \Theta ( \Delta (G(n,p)))$ for $\log n / n \ll p \leq n^{-1/2}$ whp.
        \item [(iii)] If $0 \leq p \leq O( \log n / n )$ then $G(n,p)$ has $\Omega ( \Delta (G(n,p)))$ vertices 
        of degree $\Omega ( \Delta (G(n,p)))$ whp (e.g. see Theorem 3.1 in \cite{bol-rg}). It is 
        therefore possible to find a set $U$ of $c \cdot \Delta (G(n,p))$ vertices with degree at 
        least $5|U|$, provided that $c>0$ is sufficiently small. 
        \item [(iv)] It is also true that if 
        $p \leq n^{-3/4}$ then whp $|N(u) \cap N(u')| \leq 3$ for all pairs of distinct vertices $u, u' \in V(G(n,p))$. With
        $U$ chosen as in (iii) it follows 
        that $|N(u) \setminus (U \cup \cup _{u' \in U\setminus \{u'\}} N(u'))| \geq 
        |N(u)| - |U| - 3|U| \geq |U|$. Thus $f(G(n,p)) = \Theta (\Delta (n,p))$ for $p = O(\log n / n)$.
    \end{itemize}

\section{Concluding remarks}
\label{section: concluding-remarks}

Theorem \ref{thm: N-T-intro-version} proves an essentially sharp dependence between $\hom (G)$ and $f(G)$ for $n$-vertex graphs with $\hom (G) \geq n^{1/2}$, which asymptotically resolves a conjecture of Narayanan and Tomon from \cite{narayanan}. It would be appealing to further remove the logarithmic terms here.

Another perhaps more compelling problem is to understand the relationship between these parameters when $\hom (G) < n^{1/2}$. Recall that Theorem \ref{thm: DD for G(n,p)} gives:
$$
f\big(G(n,p)\big)= 
\begin{cases}
      \Theta\left(\sqrt[3]{pn^2}\right) \text{ for } p \in [n^{-1/2}, 1/2];\\
      \Theta \big (\Delta (G(n,p)) \big ) \text{ for } p \in [0,  n^{-1/2}].
    \end{cases}
$$
\indent It is well known that $\text{hom}(G(n,p))\sim -\log n/\log (1-p)$ (see e.g. \cite{cliques}) when $0<p\leq 1/2$ is a fixed constant. For a general $p:=p(n)\leq 1/2$, the probability of having a set of size $k$ which is homogeneous in $G(n,p)$ is at most: 
$$
\binom {n}{k} \big ( p^{\binom {k}{2}} + (1-p)^{\binom {k}{2}} \big )
    \leq 
2 n^k (1-p)^{\binom {k}{2}} \leq 2 n^k e^{-p \binom {k}{2}}
= 2 \left(n e^{-p(k-1)/2} \right)^{k}.
$$
\indent In particular, $\mbox {hom}(G(n,p)) \leq 4p^{-1} \log n$ whp. Combined with the bounds for $f(G(n,p))$ from Theorem \ref{thm: DD for G(n,p)} we find that for $p \in [n^{-1/2}, 1/2]$ we have: 
$$
    f(G(n,p)) = \widetilde \Omega \Bigg ( \sqrt [3] {\frac { n^2 } {\hom (G(n,p)) }} \Bigg ) \quad \mbox{whp}.
$$
\indent We believe that a similar bound holds for any $n$-vertex graph $G$ with $\hom (G) < n^{1/2}$.

\begin{conjecture}
\label{conj: other-regime}
If $G$ is an $n$-vertex graph then: 
    \begin{align*}
        f(G) \geq \min \left (   \sqrt [3] { \frac {n^2}{\hom (G)}}   , \frac {n}{\hom (G) } \right ) n^{-o(1)}.
    \end{align*}
\end{conjecture}

Observe that the minimum above changes exactly when $\hom (G) = n^{1/2}$, value after which the Narayanan--Tomon conjecture begins to apply. Theorem \ref{thm: N-T-intro-version} proves it for $\hom (G) \geq n^{1/2}$. Theorem \ref{thm: DD for G(n,p)} shows that this behaviour is essentially tight for $G(n,p)$ when $p = n^{-1/2}$, when $\hom (G(n,p)) = n^{1/2+o(1)}$. At the opposite extreme, $n$-vertex graphs with $\hom (G)$ as small as possible (Ramsey graphs) were proven by Jenssen et al. in \cite{JKLY} to have $f(G) = \Omega (n^{2/3})$, and so the conjecture is true at both ends of the interval $\hom (G) \in [\Omega (\log n), n^{1/2}]$.\vspace{2mm}

\noindent \textbf{Acknowledgement} 

We would like to thank the referees for their careful reading of the paper, and for a number of useful suggestions.

\printbibliography

\end{document}